\documentclass[oneside, a4paper,12pt]{amsart}

\usepackage[T1]{fontenc}
\usepackage[ansinew]{inputenc}

\usepackage{geometry}

\usepackage{amsmath}
\usepackage{amsfonts}
\usepackage{amsthm}
\usepackage{amssymb}
\usepackage{graphicx}
\usepackage{tikz}

\usepackage{pict2e}
\usepackage{xy,amscd}
\xyoption{all}

\newcommand\leftmapsto{\mathrel{\reflectbox{\ensuremath{\mapsto}}}}

\makeindex

\usepackage{aliascnt}
\usepackage{hyperref}
\hypersetup{ 
    colorlinks=false,
    pdfborder={0 0 0},}
\newcommand{\nref}[1]{\hyperref[#1]{\ref*{#1}}}

\usepackage{array} %centerd table cells m{}m{}
\newcommand{\subtile}[1] %subtitle
{
	\vspace{-0.3cm}
	\begin{center}
 		{{\textsc{#1}}}\\
	\end{center}
	\vspace{0.1cm}
}

\newcommand{\imageplain}[2] %Image scaled
{
      {\includegraphics[scale=#1]{#2.png}}
}
\newcommand{\image}[2] %Image scaled centered
{
      \begin{center} %Image
	\includegraphics[scale=#1]{#2.png}
      \end{center}
}
\DeclareMathOperator{\ord}{ord}

\newcommand{\NN }{\mathbb{N}}
\newcommand{\RR }{\mathbb{R}}

\newcommand{\ZZ }{\mathbb{Z}}

\newcommand{\id }{\mathrm{id}}

\newcommand{\Ac }{\mathcal{A}}

\newcommand{\Kc }{\mathcal{K}}

\newcommand{\Cc }{\mathcal{C}}

\DeclareMathOperator{\Aut}{Aut}

\DeclareMathOperator{\Hom}{Hom}

\DeclareMathOperator{\ad}{ad}

\newcommand{\Cm }{C}
\newcommand{\cm }{c}
\newcommand{\rfl }{\rho }
\newcommand{\s }{\sigma }
\newcommand{\Wg }{\mathcal{W}}
\newcommand{\Ob }{\mathrm{Ob}}
\newcommand{\re }{^\mathrm{re}}
\newcommand{\rer }[1]{(R\re)^{#1}}
\newcommand{\rsC }{\mathcal{R}}

\newcommand{\BDelta}{\Delta}

\newcommand{\B}{\mathcal{B}} 	% Nichols Algebra
\newcommand{\g}{\mathfrak{g}} 	% arbitrary Lie algebra
\renewcommand{\sl}{\mathfrak{sl}}	% sl Lie algebra
\newcommand{\Z}{\mathbb{Z}}  	% Integers
\newcommand{\C}{\mathbb{C}}  	% Complex
\newcommand{\D}{\mathbb{D}}  	% Dihedral
\renewcommand{\S}{\mathbb{S}}  	% Symmetric
\newcommand{\CC}{\mathcal{C}}   % Bicategory
\newcommand{\YD}[1]{\;_{#1}^{#1}\mathcal{YD}}

\newcommand{\scheme}{graph }
\newcommand{\schemen}{graph}
\newcommand{\schemes}{graphs }
\newcommand{\schemesn}{graphs}

\theoremstyle{plain}
\newtheorem{theorem}{Theorem}[section]
\newtheorem*{theoremX}{Theorem}

\newtheorem{corollary}[theorem]{Corollary}

\newtheorem{lemma}[theorem]{Lemma}

\theoremstyle{remark}
\newtheorem{definition}[theorem]{Definition}

\newtheorem{example}[theorem]{Example}
\newtheorem*{exampleX}{Example}

\newtheorem{question}[theorem]{Question}
\newtheorem{remark}[theorem]{Remark}

\usepackage{enumerate}

\title[A simplicial complex of Nichols algebras]
{A simplicial complex of Nichols algebras}

\author{M.~Cuntz}
\address{Michael Cuntz,
Institut f\"ur Algebra, Zahlentheorie und Diskrete Mathematik,
Fakult\"at f\"ur Mathematik und Physik,
Leibniz Universit\"at Hannover,
Welfengarten 1,
D-30167 Hannover, Germany}
\email{cuntz@math.uni-hannover.de}

\author{S.~Lentner}
\address{Simon Lentner,
Fachbereich Mathematik, Bereich Algebra und Zahlentheorie, University Hamburg}
\email{simon.lentner@uni-hamburg.de}
\thanks{S.~Lentner is partly supported by the DFG Research Training Group 1670. Most results of this article were achieved at meetings supported by the DFG within the priority programme 1388.}

\begin{document}

\begin{abstract}
We translate the concept of restriction of an arrangement in terms of Hopf 
algebras. In consequence, every Nichols algebra gives rise to a simplicial 
complex decorated by Nichols algebras with restricted root systems.
As applications, some of these Nichols algebras provide Weyl groupoids which do 
not arise for diagonal Nichols algebras and in fact we 
realize all root systems of finite Weyl groupoids of rank greater than three. 
Further, our result explains the root systems of the folded Nichols algebras 
over nonabelian groups and of generalized Satake diagrams.
\end{abstract}

\maketitle

\setcounter{tocdepth}{2}
\tableofcontents

\section{Introduction}

Nichols algebras are Hopf algebras in a braided category, that enjoy certain universal properties and appear naturally in every pointed Hopf algebra. The most prominent examples are the Borel parts $u_q(\g)^\pm$ of the small quantum groups. In fact the Lie-theoretical flavour is retained in general, and we know from \cite{AHS10} that any Nichols algebra is controlled by a Weyl groupoid. This is a natural generalization of a Weyl group, where different Dynkin diagrams are attached to the different groupoid objects. Since \cite{p-C10} we have a good perception of a Weyl groupoid as an arrangement of hyperplanes, and in \cite{p-CH10} all finite Weyl groupoids were classified: There are infinitely many finite Weyl groupoids of rank $2$, an additional infinite series between $D_n$ and $C_n$, as well as $74$ sporadic cases up to rank $8$.

However, compared to the theory of semisimple Lie algebras where root systems provide a complete classification (at least over the complex numbers), for Nichols algebras the situation is more complicated: There are non-isomorphic Nichols algebras whose corresponding Weyl groupoids are equivalent, and it is not known yet whether all Weyl groupoids arise as symmetry structures of Nichols algebras (not even in the case of finite Weyl groupoids).
The latter problem has been brought up on several occasions, notably by the first author at the Oberwolfach Mini-Workshop on ``Nichols Algebras and Weyl Groupoids'' in 2012:
\begin{question}\label{q_Cuntz}
Is there a Weyl groupoid which does not occur as symmetry structure of a Nichols algebra?\footnote{For Nichols algebras over groups, not all Weyl groupoids appear.} Is there a characterization of those root systems that cannot appear as root systems of Nichols algebras?
\end{question}

One way to answer part of these questions is to introduce constructions producing new Nichols algebras (and hopefully new Weyl groupoids as well).
For example, \emph{folding of root systems} has been investigated by the second author to
construct new large rank Nichols algebras over nonabelian groups \cite{Len14a}. As proven recently by Heckenberger and Vendramin \cite{HV14}, these Nichols algebras are the only examples in
large rank and characteristic zero. There are four exceptions of rank two and three which are not
obtained by this folding construction, and for rank one the classification is still largely open. The second author hence asked at the Oberwolfach Mini-Workshop on ``Infinite dimensional Hopf algebras'' in 2014:
\begin{question}\label{q_Lentner}
Systematically explain the impact of the folding contruction on the root system\footnote{This has previously been calculated by hand.}. Are there more general folding constructions on the root system of a Nichols algebra? Is there a closed construction of all Nichols algebras over nonabelian groups in this way? 
\end{question}

In this article we discuss such a construction, which partly answers the previous questions, namely \emph{restriction} on the arrangement of hyperplanes given by the root system. More precisely we proceed as follows:\\

The set $\Ac$ of hyperplanes $\alpha^\perp\subset V$ orthogonal to the roots is a \emph{crystallographic arrangement} and this induces an equivalence between crystallographic arrangements and finite root systems. For details see the preliminaries in Section \ref{sec_RootSystems}.

Let $X$ be a subspace of $V$. Then the restriction $\Ac^X$ is the set of hyperplanes in $\Ac$ of the form $X\cap H$ for $H\in\Ac$. This construction is discussed in Section 3. Typically, not all hyperplanes $H\in\Ac$ give rise to hyperplanes $X\cap H$ in $\Ac^X$, and several hyperplanes in $\Ac$ may give rise to the same hyperplane in $\Ac^X$.\\ 
There are two important special cases:\\
% \begin{itemize}
% \item 
$\bullet$ The case when $X$ is an intersection of some hyperplanes of $\Ac$, we call this \emph{parabolic restriction}. Then $\Ac^X$ is automatically again crystallographic. In fact, most finite Weyl groupoids appear in this way, including examples which are not attained from Nichols algebras over groups.
\\
% \item 
$\bullet$ The case when $X$ is the fixpoint set of an automorphism $g$ of $\Ac$. In general, $\Ac^X$ must not be crystallographic, but in several cases it is. In the special case where $g$ permutes a set of simple roots (called \emph{permutation restriction}), these restrictions describe the root system of the folded Nichols algebra over the centrally extended group.\\
% \end{itemize}

In Section \ref{sec_arrnich} we turn our attention to Weyl groupoids of Nichols 
algebras. We start with a Nichols algebra $\B(M)$, which is typically diagonal.
Let $J\subset I$ be a subset of the simple roots for
$\B(M)$, denote by $M_J\subset M$ the corresponding sub-Yetter-Drinfel'd module 
and consider the associated Nichols algebra of coinvariants 
$$\B(\bar{M}):=\B(M)^{coin(\B(M_J))}.$$
This is a Nichols algebra in the category of $\B(M_J)$-Yetter-Drinfel's modules. The key result of this article in Theorem \nref{thm_NicholsAlgebraRestrictionPBW} is that the root system of
$\B(\bar{M})$ is the parabolic restriction $\Ac^X,\;X=J^\perp$.  Moreover, 
$\B(M_J)$ (determining the category) corresponds precisely to the sub-arrangement of hyperplanes in $\Ac$ that do not give rise to hyperplanes in $\Ac^X$, while the number of different hyperplanes in $\Ac$ that give rise to the same hyperplane in $\Ac^X$, determines the dimensions of the Yetter-Drinfel'd modules in the PBW-basis of $\B(\bar{M})$.\\

In Section \ref{simpcomp_section} we conclude by viewing the set of all parabolic restrictions of a Nichols algebra as a simplicial complex. In this picture, the reflection operation for Nichols algebras has a particularly nice interpretation in terms of restrictions and dualization.\\

\noindent
We now discuss \emph{applications}, most of which are new and interesting 
already for the case of a diagonal Nichols algebra and even for the special 
case of a quantum group $u_q(\g)^+$ associated to a Lie algebra $\g$:\\

As first result, we obtain Nichols algebras with Weyl groupoids, which were not 
attained from a previously known Nichols algebra over a group:\\

\begin{theoremX}[\nref{thm_newWeylGroupoid}]~ 
\begin{enumerate}
\item There exist Nichols algebras whose corresponding crystallographic arrangements are the sporadic arrangements of rank three labeled $7,13,14,15,20,23$, although these Nichols algebras do not exist over any finite group.
\item Since every crystallographic arrangement of rank greater than three is a restriction of a Weyl arrangement, every crystallographic arrangement of rank greater than three is symmetry structure of some Nichols algebra.
\end{enumerate}
\end{theoremX}

\begin{exampleX}[\nref{exm_newWeylGroupoid}]
Let $\B(M)=u_q(E_7)^+$, which is a diagonal Nichols algebra of dimension 
$\ord(q^2)^{63}$. We consider the parabolic restriction indicated in the 
diagram:
\begin{center}
\includegraphics[scale=.6]{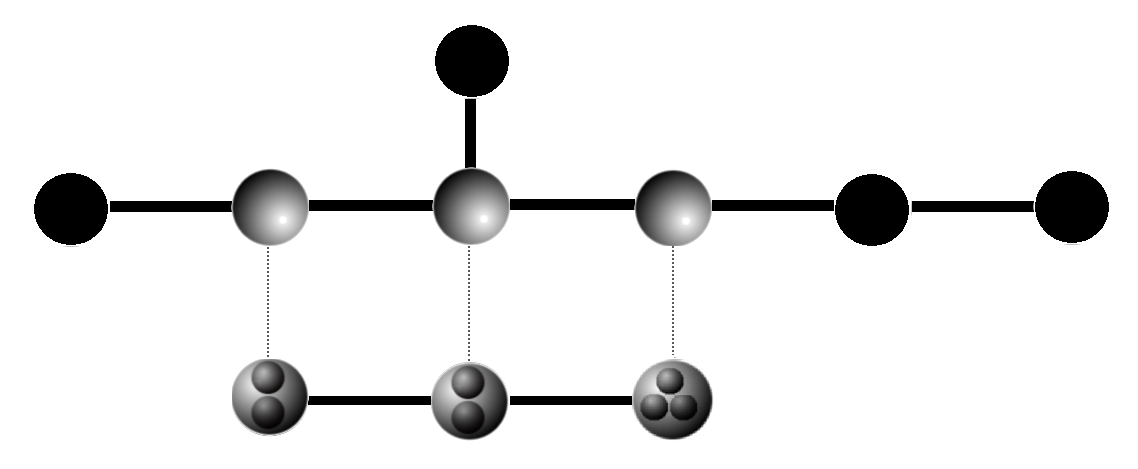}
\end{center}

Then the restriction of the Weyl group arrangement $E_7$ with $63$ roots has a 
Weyl groupoid of sporadic type $7$ and $13$ roots (one of which is nonreduced), 
namely\\

\begin{center}
\begin{tabular}{l|lllllllllllll}
$R_+^a$ 				&	
$\bar{\alpha}_3$&$\bar{\alpha}_4$	&$\bar{\alpha}_5$ &	$(1,1,0)$ & 
$(0,1,1)$	&
$(1,1,1)$ & $+2(1,1,1)$	\\
multiplicity 		& $2$							
& $2$							& $3$				
			&	$4$				& $6$ 			
& 
$12$ & +3\\
\hline
$R_+^a$ 				& $(1,2,1)$	& $(2,2,1)$	& 
$(1,2,2)$	& $(2,3,2)$	& $(2,3,3)$	& $(2,4,3)$&$(3,4,3)$\\
multiplicity 		&	$6$				& $3$			
	& $6$				& $6$				& $2$		
		& $1$ 			& $2$\\
\end{tabular}
\end{center}~\\

This yields a Nichols algebra $\B(\bar{M})$ of dimension $\ord(q^2)^{58}$ over 
an object $\bar{M}$ of rank $3$ and dimension $2+2+3$ in the braided category 
$\YD{\B(M_J)}$ with $J=\{\alpha_1,\alpha_2,\alpha_6,\alpha_7\}$, where 
$\B(M_J)$ 
is an ordinary Borel part $u_q(A_1\times A_1\times A_2)^+$.\\

This root system is clearly not one of Lie type. This means that reflections in 
the Weyl groupoid connect the Weyl chamber described above with other chambers 
corresponding possibly to different Dynkin diagrams and non-isomorphic Nichols 
algebras (compare Example \nref{exm_CycleNicholsAlgebra}).\\

For the restriction these different choices arise from different parabolics 
$J'$ of $E_7$ in the Weyl group orbit of the choice $J$ above. In 
the main text, this is discussed in more detail.
\end{exampleX}

As a second application, we note that quantum groups $u_q(\g)$ (and other 
pointed Hopf algebras \cite{AY13}) can be obtained from a Drinfel'd double 
construction. Our restricted Nichols algebra $\B(\bar{M})$ exists in the 
category $\YD{\B(M_J)}$, hence it is roughly a Nichols algebra in the category 
of $u_q(\g_J)$-modules, where $\g_J$ is the Lie subalgebra of $\g$ associated to 
the subset of simple roots $J$.
 As an example\footnote{We are thankful to N.\ Andruskiewitsch for pointing out 
this source to us.} with $J=I\backslash\{\alpha_1\}$ of type $A_{n-1}$, this 
explains the $\sl_n$-representations in \cite[Sec.\ 4]{Rosso98}. See our 
Example \nref{exm_RestrictionA2}.

\begin{question}
	Can Nichols algebras $\B(\bar{M})$ in the category of modules over a 
quantum double of some $\B(M_J)$ and in a second step over $u_q(\g_J)$ be 
characterized by restrictions of ordinary Nichols algebras $\B(M),M\supset M_J$? 
This would be the first Nichols algebras over non-semisimple Hopf algebras, and 
it would also be interesting from the perspective of logarithmic conformal field 
theory.
\end{question}

A third application is in the context of quantum symmetric pairs: Our graphical representation using black dots is suggested by \emph{Satake diagrams} of Lie algebras\footnote{We thank S.\ Kolb for having brought up this topic to our attention at the Oberwolfach miniworkshop mentioned above.}. Indeed, we define \emph{folding} of a crystallographic arrangement (resp.\ the root system of a Weyl groupoid) as the restriction to the fixed-point set $X$ of some automorphism group of the root system, and we prove a respective Lemma \nref{lm_folding} that enables us to classify such scenarios again by Satake diagrams. 

\begin{example}
	The previous example shows that we even get new Satake diagrams for ordinary Lie algebras, here $E_7$ over a parabolic subalgebra $A_1\times A_1\times A_2$. This would be excluded by the classical theory, because the restricted root system (relative root system, quotient root system) corresponds to a Weyl groupoid.
	\end{example}

Satake diagrams for Nichols algebras of Lie type appear prominently in the theory of quantum symmetric pairs, introduced by Noumi, Sugitani, and Dijkhuizen for $\g$ of classical type in connection with the reflection equation and in theoretical physics (see e.g. \cite{NS95}) and independently by G.\ Letzter from the Iwasawa decomposition for the quantum group (see e.g.\ \cite{Let97}), generalized to Kac Moody algebras by \cite{Kolb14}. Our results show that one can consider more general Satake diagrams for Lie type root systems, which do not yield Lie type root systems anymore, as in the example above. On the other hand they allow Satake diagrams to be considered for arbitrary Nichols algebras, in particular for super Lie algebras and color Lie algebras. \\

As a final outlook, this may point to a more general folding construction for Nichols algebras than in \cite{Len14a}, since in the known cases the restricted root system $\Ac^X$ reproduces the correct folded root system. As a new Example \nref{exm_newNicholsAlgebra} we identify an inclusion chain of three Nichols algebra, for which restriction would reproduce the correct root system and root space dimensions for the exceptional Nichols algebras of rank $3,2,1$ over nonabelian groups involving $\S_3$ in \cite{HV14}:
\begin{center} 
		\includegraphics[scale=.8]{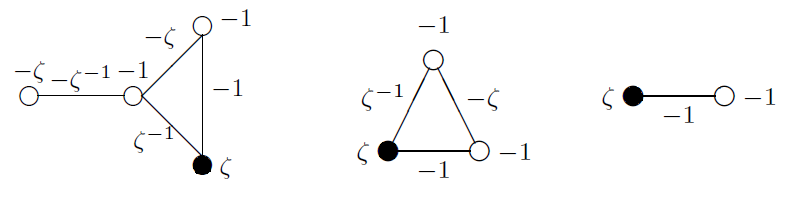}
\end{center}
\enlargethispage{1cm}
In this article we do \emph{not} attempt to define non-parabolic restriction of Nichols algebras, that should come with an additional group action on the root spaces $M_\alpha$. Note however that in the theory of quantum symmetric pairs there are explicit methods to construct coideal subalgebras corresponding to the fixed-point set of the root system automorphism. This might give the right inspiration for constructing the respective restricted Nichols algebra in the general (non-parabolic, non-permutation) case.

\section{Weyl groupoids and crystallographic arrangements}\label{sec_RootSystems}

\subsection{Simplicial arrangements}
Let $r\in\NN$, $V:=\RR^r$. For $\alpha\in V^*$ we write
$\alpha^\perp = \ker(\alpha)$.
We first recall the definition of a simplicial arrangement (compare \cite[1.2, 5.1]{OT}).
\begin{definition}\label{A_R}
An \emph{arrangement of hyperplanes} $\Ac$ is a finite set of hyperplanes in $V$.
Let $\Kc(\Ac)$ be the set of connected components ({\it chambers}) of $V\backslash \bigcup_{H\in\Ac} H$.
If every chamber $K$ is an {\it open simplicial cone}, i.e.\ there exist
$\alpha^\vee_1,\ldots,\alpha^\vee_r \in V$ such that
\begin{equation*}
K = \Big\{ \sum_{i=1}^r a_i\alpha^\vee_i \mid a_i> 0 \quad\mbox{for all}\quad
i=1,\ldots,r \Big\} =: \langle\alpha^\vee_1,\ldots,\alpha^\vee_r\rangle_{>0},
\end{equation*}
then $\Ac$ is called a {\it simplicial arrangement}.
\end{definition}

\begin{example}~
\begin{enumerate}
\item The picture on the left is a simplicial arrangement in $\RR^2$. The picture on the right is a representation of a three dimensional simplicial arrangement in the projective plane. The simplicial cones become triangles in this representation.
%\begin{figure}
\begin{center}
\setlength{\unitlength}{0.2pt}
\begin{picture}(400,380)(100,180)
\moveto(300.000000000000000000000000000,600.000000000000000000000000000)
\lineto(300.000000000000000000000000000,200.000000000000000000000000000)
\moveto(368.404028665133746608819922937,212.061475842818323189178144535)
\lineto(231.595971334866253391180077064,587.938524157181676810821855465)
\moveto(428.557521937307865264528681981,246.791111376204392959521469889)
\lineto(171.442478062692134735471318019,553.208888623795607040478530111)
\moveto(473.205080756887729352744634151,300.000000000000000000000000000)
\lineto(126.794919243112270647255365849,500.000000000000000000000000000)
\moveto(496.961550602441611873348604918,365.270364466613930229656674646)
\lineto(103.038449397558388126651395082,434.729635533386069770343325354)
\moveto(496.961550602441611873348604918,434.729635533386069770343325354)
\lineto(103.038449397558388126651395082,365.270364466613930229656674646)
\moveto(473.205080756887729352744634151,500.000000000000000000000000000)
\lineto(126.794919243112270647255365849,300.000000000000000000000000000)
\moveto(428.557521937307865264528681981,553.208888623795607040478530111)
\lineto(171.442478062692134735471318019,246.791111376204392959521469889)
\moveto(368.404028665133746608819922937,587.938524157181676810821855465)
\lineto(231.595971334866253391180077064,212.061475842818323189178144535)
\strokepath
\end{picture} \quad\quad\quad\quad
\setlength{\unitlength}{0.5pt}
\begin{picture}(130,180)(240,480)
\moveto(231.120000000000000000000000000,546.960000000000000000000000000)
\lineto(386.880000000000000000000000000,546.960000000000000000000000000)
\moveto(230.160000000000000000000000000,587.520000000000000000000000000)
\lineto(387.600000000000000000000000000,587.520000000000000000000000000)
\moveto(329.520000000000000000000000000,645.360000000000000000000000000)
\lineto(329.760000000000000000000000000,488.400000000000000000000000000)
\moveto(289.200000000000000000000000000,645.360000000000000000000000000)
\lineto(289.200000000000000000000000000,488.400000000000000000000000000)
\moveto(264.960000000000000000000000000,635.280000000000000000000000000)
\lineto(336.960000000000000000000000000,491.280000000000000000000000000)
\moveto(281.280000000000000000000000000,491.280000000000000000000000000)
\lineto(352.560000000000000000000000000,634.080000000000000000000000000)
\moveto(281.280000000000000000000000000,643.200000000000000000000000000)
\lineto(353.280000000000000000000000000,499.200000000000000000000000000)
\moveto(264.240000000000000000000000000,498.480000000000000000000000000)
\lineto(336.960000000000000000000000000,643.200000000000000000000000000)
\moveto(241.920000000000000000000000000,611.040000000000000000000000000)
\lineto(385.200000000000000000000000000,539.040000000000000000000000000)
\moveto(233.280000000000000000000000000,539.040000000000000000000000000)
\lineto(376.800000000000000000000000000,611.040000000000000000000000000)
\moveto(241.920000000000000000000000000,523.440000000000000000000000000)
\lineto(385.200000000000000000000000000,594.720000000000000000000000000)
\moveto(233.280000000000000000000000000,594.720000000000000000000000000)
\lineto(376.800000000000000000000000000,522.720000000000000000000000000)
\moveto(284.160000000000000000000000000,643.680000000000000000000000000)
\lineto(334.560000000000000000000000000,490.800000000000000000000000000)
\moveto(232.320000000000000000000000000,541.920000000000000000000000000)
\lineto(385.200000000000000000000000000,592.560000000000000000000000000)
\moveto(228.240000000000000000000000000,567.120000000000000000000000000)
\lineto(390.240000000000000000000000000,567.360000000000000000000000000)
\moveto(309.360000000000000000000000000,648.240000000000000000000000000)
\lineto(309.360000000000000000000000000,486.240000000000000000000000000)
\moveto(250.800000000000000000000000000,625.200000000000000000000000000)
\lineto(367.200000000000000000000000000,508.800000000000000000000000000)
\moveto(250.800000000000000000000000000,508.800000000000000000000000000)
\lineto(367.200000000000000000000000000,625.200000000000000000000000000)
\strokepath
\end{picture}
%caption{Simplicial arrangements in $\RR^2$ and $\RR^3$.}
\end{center}
%\label{fig1}
%\end{figure}
\item Let $W$ be a real reflection group, $R$ the set of roots of $W$. Then
$\Ac = \{\alpha^\perp\mid \alpha\in R\}$ is a simplicial arrangement.
\end{enumerate}
\end{example}

\subsection{Crystallographic arrangements}\label{cryarr}

Let $\Ac=\{H_1,\ldots,H_n\}$, $|\Ac|=n$ be a simplicial arrangement.
For each $H_i$, $i=1,\ldots,n$ we choose an element $x_i\in V^*$ such that $H_i=x_i^\perp$
and let $R := \{\pm x_1,\ldots,\pm x_n\}\subseteq V^*$.

Since $\Ac$ is simplicial, for each chamber $K\in\Kc(\Ac)$ there is a unique choice of a basis $\{\alpha^\vee_1,\ldots,\alpha^\vee_r\}$ such that $K=\langle \alpha^\vee_1,\ldots,\alpha^\vee_r \rangle_{>0}$ and such that its dual basis $\{\alpha_1,\ldots,\alpha_r\}$ is a subset of $R$.
We write
\begin{eqnarray*}
\BDelta^K &=& \{\alpha_1,\ldots,\alpha_r\}\\
&=& \{\mbox{ ``normal vectors in $R$ of the walls of } K \mbox{ pointing to the inside'' }\}.
\end{eqnarray*}

% For each chamber $K\in\Kc(\Ac)$ set
% \begin{eqnarray*}
% \BDelta^K &=& \{\mbox{ normal vectors in $R$ of the walls of } K \mbox{ pointing to the inside }\} .
% \end{eqnarray*}
% If $\alpha^\vee_1,\ldots,\alpha^\vee_r$ is the dual basis to
% $\BDelta^K=\{\alpha_1,\ldots,\alpha_r\}$, then
% $K=\langle \alpha^\vee_1,\ldots,\alpha^\vee_r \rangle_{>0}$
% since $\Ac$ is simplicial.

We are now ready for the main definition.

\begin{definition}
Let $\Ac$ be a simplicial arrangement and $R\subseteq V^*$ a finite set
such that $\Ac = \{ \alpha^\perp \mid \alpha \in R\}$ and $\RR\alpha\cap R=\{\pm \alpha\}$
for all $\alpha \in R$.
We call $(\Ac,R)$ a {\it crystallographic arrangement} if
for all $K\in\Kc(\Ac)$:
\[ R \subseteq \sum_{\alpha \in \BDelta^K} \ZZ \alpha. \]
Two crystallographic arrangements $(\Ac,R)$, $(\Ac',\bar{R})$ in $V$ are called {\it equivalent}
if there exists $\psi\in\Aut(V^*)$ with $\psi(R)=\bar{R}$. We then write $(\Ac,R)\cong(\Ac',\bar{R})$.
\end{definition}

\begin{example}
\begin{enumerate}
\item Let $R$ be the set of roots of the root system
of a crystallographic Coxeter group. Then
$(\{\alpha^\perp\mid \alpha\in R\},R)$ is a crystallographic arrangement.
\item If $R_+ := \{(1,0),(3,1),(2,1),(5,3),(3,2),(1,1),(0,1)\}$,
then\\ $(\{\alpha^\perp\mid \alpha\in R_+\},R_+\dot\cup-R_+)$ is a crystallographic arrangement.
\end{enumerate}
\end{example}

As most crystallographic arrangements are not defined by reflection groups, they have less symmetries in general.
As a substitute for this large symmetry group, it turns out that a corresponding \emph{Weyl groupoid} is the right symmetry structure.

\subsection{Cartan \schemes and Weyl groupoids}\label{CSWG}

We now define the notion of a Weyl groupoid which was introduced by Heckenberger and Yamane \cite{p-HY-08} and reformulated in \cite{p-CH09a}.
But before we present the axioms, let us look at an example.

\begin{example}\label{cswg_ex}
Let $\alpha_1=(1,0,0)$, $\alpha_2=(0,1,0)$, $\alpha_3=(0,0,1)$, and
$$R^a_+:=\{\alpha_1,\alpha_2,\alpha_3,(0,1,1),(0,1,2),(1,0,1),(1,1,1),(1,1,2)\}.$$
For $1\le i,j\le 3$, define entries $\cm_{i,j}$ of a matrix $C$ by
\[ \cm_{i,j}:=-\max\{ k \mid k\alpha_i+\alpha_j \in R^a_+ \},
\quad\quad \cm_{i,i}:=2\quad\quad\mbox{for}\quad i\ne j, \]
thus
\[ C^a = (c_{i,j})_{i,j} =
\begin{pmatrix} 2&0&-1\\0&2&-1\\-1&-2&2 \end{pmatrix}. \]
This is a {\it generalized Cartan matrix}. It defines {\it reflections} via
\[  \s _i (\alpha_j) = \alpha_j - \cm _{ij} \alpha_i \qquad
    \text{for $j=1,2,3$.} \]
For instance,
\[ \s_1 = \begin{pmatrix} -1&0&1\\0&1&0\\0&0&1 \end{pmatrix}\]
and
\[ \s_1(R^a_+)=\{-\alpha_1,\alpha_2,(1,0,1),(1,1,1),(2,1,2),\alpha_3,(0,1,1),(1,1,2)\}. \]
The elements of $\s_1(R^a_+)$ are {\it positive} or {\it negative}.
Let $R^a=R^a_+\cup -R^a_+$ and
$R^b = \s_1(R^a) =: R^b_+ \cup -R^b_+$.
Again, one can construct a Cartan matrix from $R^b_+$ and it gives
new reflections. In this example, we obtain a diagram:
\begin{tiny}
\[ \xymatrix{
{\left( \begin{array}{ccc}
 2&0&-1\\0&2&-1\\-1&-1&2
\end{array} \right)} \ar@{-}[r]^{\sigma_3}
&
{\left( \begin{array}{ccc}
 2&-1&-1\\-1&2&-1\\-1&-1&2
\end{array} \right)} \ar@{-}[r]^{\sigma_2}
&
{\left( \begin{array}{ccc}
 2&-1&0\\-1&2&-1\\0&-1&2
\end{array} \right)} \ar@{-}[d]^{\sigma_1}
\\
{\left( \begin{array}{ccc}
 2&0&-1\\0&2&-1\\-1&-2&2
\end{array} \right)} \ar@{-}[u]^{\sigma_1}
& &
{\left( \begin{array}{ccc}
 2&-1&0\\-1&2&-2\\0&-1&2
\end{array} \right)}
} \]
\end{tiny}
\end{example}
For the general definition, we first recall:

\begin{definition}
Let $I:=\{1,\ldots,r\}$ and
$\{\alpha_i\,|\,i\in I\}$ the standard basis of $\ZZ ^I$.
% By \cite[\S 1.1]{b-Kac90} 
A {\it generalized Cartan matrix}
$\Cm =(\cm _{ij})_{i,j\in I}$
is a matrix in $\ZZ ^{I\times I}$ such that
\begin{enumerate}
\item[(M1)] $\cm _{ii}=2$ and $\cm _{jk}\le 0$ for all $i,j,k\in I$ with
  $j\not=k$,
\item[(M2)] if $i,j\in I$ and $\cm _{ij}=0$, then $\cm _{ji}=0$.
\end{enumerate}
\end{definition}

The above diagram of matrices is a \emph{Cartan \schemen}\footnote{In earlier papers, Cartan graphs were called Cartan schemes. The new term was chosen by Andruskiewitsch, Heckenberger and Schneider to avoid confusion with geometric schemes.}:

\begin{definition}
Let $A$ be a non-empty set, $\rfl _i : A \to A$ a map for all $i\in I$,
and $\Cm ^a=(\cm ^a_{jk})_{j,k \in I}$ a generalized Cartan matrix
in $\ZZ ^{I \times I}$ for all $a\in A$. The quadruple
\[ \Cc = \Cc (I,A,(\rfl _i)_{i \in I}, (\Cm ^a)_{a \in A})\]
is called a \textit{Cartan \schemen} if
\begin{enumerate}
\item[(C1)] $\rfl _i^2 = \id$ for all $i \in I$,
\item[(C2)] $\cm ^a_{ij} = \cm ^{\rfl _i(a)}_{ij}$ for all $a\in A$ and
  $i,j\in I$.
\end{enumerate}
\end{definition}

\begin{definition}
Let $\Cc = \Cc (I,A,(\rfl _i)_{i \in I}, (\Cm ^a)_{a \in A})$ be a
Cartan \schemen. For all $i \in I$ and $a \in A$ define $\s _i^a \in
\Aut(\ZZ ^I)$ by
\begin{align}
\s _i^a (\alpha_j) = \alpha_j - \cm _{ij}^a \alpha_i \qquad
\text{for all $j \in I$.}
\label{fwg_eq:sia}
\end{align}
The \textit{Weyl groupoid of} $\Cc $
is the category $\Wg (\Cc )$ such that $\Ob (\Wg (\Cc ))=A$ and
the morphisms are compositions of maps
$\s _i^a$ with $i\in I$ and $a\in A$,
where $\s _i^a$ is considered as an element in $\Hom (a,\rfl _i(a))$.
The cardinality of $I$ is the \textit{rank of} $\Wg (\Cc )$.
\end{definition}
\begin{definition}
A Cartan \scheme is called \textit{connected} if its Weyl grou\-poid
is connected, that is, if for all $a,b\in A$ there exists $w\in \Hom (a,b)$.
The Cartan \scheme is called \textit{simply connected},
if $\Hom (a,a)=\{\id ^a\}$ for all $a\in A$.
There is a straight forward notion of \emph{equivalence} of Cartan \schemes which we skip here.
\end{definition}

Let $\Cc $ be a Cartan \schemen. For all $a\in A$ let
\[ \rer a=\{ \id ^a \s _{i_1}\cdots \s_{i_k}(\alpha_j)\,|\,
k\in \NN _0,\,i_1,\dots,i_k,j\in I\}\subseteq \ZZ ^I.\]
The elements of the set $\rer a$ are called \textit{real roots} (at $a$).
The pair $(\Cc ,(\rer a)_{a\in A})$ is denoted by $\rsC \re (\Cc )$.
A real root $\alpha\in \rer a$, where $a\in A$, is called positive
(resp.\ negative) if $\alpha\in \NN _0^I$ (resp.\ $\alpha\in -\NN _0^I$).

\begin{definition}
Let $\Cc =\Cc (I,A,(\rfl _i)_{i\in I},(\Cm ^a)_{a\in A})$ be a Cartan
scheme. For all $a\in A$ let $R^a\subseteq \ZZ ^I$, and define
$m_{i,j}^a= |R^a \cap (\NN_0 \alpha_i + \NN_0 \alpha_j)|$ for all $i,j\in
I$ and $a\in A$. We say that
\[ \rsC = \rsC (\Cc , (R^a)_{a\in A}) \]
is a \textit{root system of type} $\Cc $, if it satisfies the following
axioms.
\begin{enumerate}
\item[(R1)]
$R^a=R^a_+\cup - R^a_+$, where $R^a_+=R^a\cap \NN_0^I$, for all
$a\in A$.
\item[(R2)]
$R^a\cap \ZZ\alpha_i=\{\alpha_i,-\alpha_i\}$ for all $i\in I$, $a\in A$.
\item[(R3)]
$\s _i^a(R^a) = R^{\rfl _i(a)}$ for all $i\in I$, $a\in A$.
\item[(R4)]
If $i,j\in I$ and $a\in A$ such that $i\not=j$ and $m_{i,j}^a$ is
finite, then
$(\rfl _i\rfl _j)^{m_{i,j}^a}(a)=a$.
\end{enumerate}
\end{definition}

% The axioms (R2) and (R3) are always fulfilled for $\rsC \re $.
The root system $\rsC $ is called \textit{finite} if for all $a\in A$ the
set $R^a$ is finite. By \cite[Prop.\,2.12]{p-CH09a},
if $\rsC $ is a finite root system
of type $\Cc $, then $\rsC =\rsC \re $, and hence $\rsC \re $ is a root
system of type $\Cc $ in that case.

\begin{remark}
If $\Cc $ is a Cartan \scheme and there exists a root system of type $\Cc $,
then $\Cc $ satisfies
\begin{itemize}
  \item [(C3)] If $a,b\in A$ and $\id \in \Hom (a,b)$, then $a=b$.
\end{itemize}
\end{remark}

\begin{example}[Lie type]
Let $\g$ be a semisimple finite-dimensional complex Lie algebra. Then this is uniquely determined (up to isomorphisms) by its corresponding root system, which is the root system of a finite Weyl group $W$. The corresponding Cartan \scheme has exactly one object $a$ where $C^a$ is the Cartan matrix of $W$. The set $R^a$ is the root system of $W$.
\end{example}

\subsection{Classification of finite Weyl groupoids}\label{sec_ClassificationWeylGroupoid}

Connected simply connected Cartan \schemes for which the real roots are a finite root system (these are also called \emph{universal finite Weyl groupoids}) are in one-to-one correspondence with crystallographic arrangements.
Under the correspondence, every chamber $K$ of an arrangement corresponds to an 
object $a$; the set $R^a$ consists of the coordinates of $R$ with respect to the basis 
$\BDelta^K$.
% where $K$ corresponds to $a$.

\begin{theorem}[see \cite{p-C10}]\label{thm_EquivalenceArrangementCartanscheme}
Let $\mathfrak A$ be the set of all crystallographic arrangements and
$\mathfrak C$ be the set of all connected simply connected Cartan \schemes for
which the real roots are a finite root system.
Then the map
\[ \mathfrak{C}/_{\cong} \rightarrow \mathfrak{A}/_{\cong}, \quad
\overline{\Cc =\Cc (I,A,(\rfl _i)_{i\in I},(\Cm ^a)_{a\in A})} \mapsto
\overline{(\{\alpha^\perp\mid \alpha \in R^a\},R^a)}, \]
where $a$ is any object of $\Cc$, is a bijection.
\end{theorem}

A series of five papers by the first author and Heckenberger culminates in the following complete classification of finite Weyl groupoids and thus equivalently of crystallographic arrangements.

\begin{theorem}[see \cite{p-CH10}]\label{thm_ClassificationArrangments}
There are exactly three families of irreducible crystallographic arrangements:
\begin{enumerate}
\item The family of rank two parametrized by triangulations of
convex $n$-gons by non-intersecting diagonals.
\item For each rank $r>2$, arrangements of type $A_r$, $B_r$, $C_r$
and $D_r$, and a further series of $r-1$ arrangements denoted $\Ac^k_r(2)$ in \cite[6.4]{OT}.
\item Further $74$ ``sporadic'' arrangements of rank $r$, $3\le r \le 8$.
\end{enumerate}
\end{theorem}

\begin{remark}
	Theorem \nref{thm_ClassificationArrangments} classifies simply-connected Cartan 
	\schemesn. Every Cartan \scheme has a simply-connected cover. For example the Cartan \scheme 
	of Lie type has a single Cartan matrix and the Weyl groupoid has a single object with 
	automorphism group the Weyl group $W$. The simply-connected cover has $|W|$ Cartan 
	matrices, all of the same type, and its Weyl groupoid has $|W|$ objects and no non-trivial automorphisms.\\
	Automorphism groups and minimal quotients of finite Weyl groupoids have all been determined in \cite{p-CH10}.
	Whether one can consider such non-simply-connected quotients depends on additional data 
	in the application.	For example, the Lie algebra $\sl_3$ resp.\ the Lie superalgebra $\sl(
	2|1)$ have the same associated arrangement $A_2$, but in the first case one usually 
	considers the Weyl group (one object), while in the latter case some roots are 
	labeled differently, so there are two types of chambers.   
\end{remark}

\section{Restrictions of arrangements and root systems}\label{sec_Restriction}

\begin{definition}[{\cite[1.12-1.14]{OT}}]
Let $\Ac$ be an arrangement in $V$.
We denote $L(\Ac)$ the set of all nonempty intersections of elements of $\Ac$.

For a subspace $X\le V$,
define a subarrangement $\Ac_X$ of $\Ac$ called the \emph{localization at} $X$ by
\[ \Ac_X = \{H\in\Ac\mid X\subseteq H\}. \]
Define an arrangement $\Ac^X$ in $X$ called the \emph{restriction to} $X$ by
\[ \Ac^X=\{X\cap H\mid H\in\Ac\backslash\Ac_X \mbox{ and } X\cap H\ne \emptyset\}.\]
\end{definition}

If $\Ac$ is a reflection arrangement of a Coxeter group $W$, then localizations 
of $\Ac$ are the reflection arrangements of parabolic subgroups of $W$. Thus 
localizations are easy to understand from the algebraic point of view. 
Restrictions $\Ac^X$ however are not reflection arrangements in general, even in 
the case when $X$ is in the intersection lattice $L(\Ac)$. 

However, the restrictions of 
a crystallographic arrangement to an element $X\in L(\Ac)$ is always again a
crystallographic arrangement. This yields an important class of examples:

\subsection{Parabolic restriction}\label{sec_ParabolicRestriction}

In this subsection we discuss the following case of restriction:

\begin{definition}\label{def_ParabolicRestriction}
	A \emph{parabolic restriction} of an arrangement $\Ac$ is a restriction $\Ac^X$ to an  		
	intersection of existing hyperplanes  $X\in L(\Ac)$.
\end{definition}

In contrast to parabolic localization $\Ac_X$ this corresponds to 
\emph{quotienting out} a parabolic subgroup and the result is in general not a 
reflection arrangement. Note that any parabolic restriction to an intersection 
of hyperplanes $X\in L(\Ac)$ can be obtained by repeatedly 
restricting to one hyperplane $H\in \Ac\subset L(\Ac)$.\\

It is an easy fact that (see \cite{p-CRT11}, \cite{p-BC10}):
 
\begin{lemma}\label{lm_ParabolicRestriction} Let $(\Ac,R)$ by a crystallographic 
	arrangement, then
	any parabolic restriction $\Ac^X$ to some $X\in L(\Ac)$ is again a crystallographic arrangement.\\
	More precisely, a root system $\bar{R}$ for $\Ac^X$ is given as follows: Suppose 
		without loss of generality $X\in \Ac$ a single hyperplane and a chamber $K$ chosen 		
		adjacient to $X$, say by suitable numbering $X=\alpha_1^\perp$. Then the restriction 
		of $\Ac$ to $X=\alpha_1^\perp$ is
		\[ \Ac^{\alpha_1^\perp} = \{ \beta'^\perp \mid \beta\in R,\:\: \beta\ne \alpha_1\}, \]
		where if $\beta = \sum_{\alpha\in \BDelta^K} b_\alpha\alpha$, then
		\[ \beta' = \frac{1}{\gcd(b_\alpha\mid \alpha_1\ne \alpha\in \BDelta^K)}\sum_{\alpha_1\ne 
		\alpha\in \BDelta^K} b_\alpha\alpha. \]
		Thus we obtain the restriction by deleting the coordinate to $\alpha_1$ and reducing 
		to the shortest vector in the lattice.  
\end{lemma}

\begin{example}\label{ex_res} Let $\alpha_1=(1,0,0)$, $\alpha_2=(0,1,0)$, $\alpha_3=(0,0,1)$.
\begin{enumerate}[a)]
\item Let
$R^a_+:=\{\alpha_1,\alpha_2,\alpha_3,(0,1,1),(1,1,0),(1,1,1)\}$,
and let $\Ac:=\{\alpha^\perp \mid \alpha \in R^a_+\}$ be the crystallographic arrangement of Lie type $A_3$. Then the restriction $\Ac^H$ of $\Ac$ to any hyperplane $H = \alpha_i^\perp$ is the crystallographic arrangement of type $A_2$ defined by
$\bar{R}^a_+ = \{(1,0),(0,1),(1,1)\}$. 
\item Let $R^a_+:=\{\alpha_1,\alpha_2,\alpha_3,(0,1,1),(0,1,2),(1,0,1),(1,1,1),(1,1,2)\}$,
and let $\Ac:=\{\alpha^\perp \mid \alpha \in R^a_+\}$ (which is not of Lie type). Then the restriction $\Ac^H$ of $\Ac$ to the hyperplane $H = \alpha_1^\perp$ is the crystallographic arrangement of Lie type $B_2$ (or $C_2$ depending on the used definition) defined by
$\bar{R}^a_+ = \{(1,0),(0,1),(1,1),(1,2)\}$.
\item Now let
\[ R^a_+ = \{ (0,0,1),(0,1,0),(0,1,1),(0,1,2),(1,0,0),(1,1,0),(1,1,1),(1,1,2),(1,2,2) \}, \]
$\Ac := \{\alpha^\perp \mid \alpha\in R^a_+\}$, thus $R^a$ is a root system of type $B_3$.
Choose $H:=(1,0,0)^\perp$. Then $\Ac^H$ is the set of kernels of
\[ \bar{R}^a_+ = \{ (0,1),(1,0),(1,1),(1,2) \}. \]
Looking more closely, we notice that there are several hyperplanes in $\Ac$ which restrict to the same element of $\Ac^H$. For example, three different roots map to the root $(1,1)$, namely (0,1,1),(1,1,1),(1,2,2). 
\end{enumerate}
\end{example}

We can see two effects at this last example:
\begin{enumerate}
\item It may be useful to keep track of the number of hyperplanes falling together under restriction. Thus instead of considering arrangements of hyperplanes, one should consider arrangements of hyperplanes with multiplicities (these are called \emph{multiarrangements}).
\item A-priori, the vectors with the deleted coordinate do not form a reduced root system, i.e.\ they may have the form $k\alpha$ for $k\in\ZZ\backslash\{\pm 1\}$ and $\alpha$ in the lattice. The reduction requires to rescale a root, for example $(1,2,2)\mapsto \frac{1}{2}(2,2)$.
\end{enumerate}
It would make sense to include the information of these two situations into the setting of Cartan \schemesn, but this would possibly make things more complicated than necessary. Notice that in the connected simply-connected case, crystallographic arrangements, Weyl groupoids, or Cartan \schemes are all uniquely determined by the roots at a single object, hence by a single set $R^a_+\subseteq \ZZ^r$.

\begin{definition}[compare \cite{p-CH10}]
We will say that a finite set $\Phi\subseteq\ZZ^r$
is a \textit{reduced root set of rank $r$} if there exists
a Cartan \scheme $\Cc$ of rank $r$ and an injective linear map $w : \ZZ^r\rightarrow\ZZ^r$
such that $w(\rer a)=\Phi$ for some object $a$.
\end{definition}

\begin{definition}
A \emph{(nonreduced) root set of rank $r$} is a set $\Phi\subseteq \ZZ^r$ such that
\[ \left\{ \frac{1}{\gcd(a_1,\ldots,a_r)}\alpha \mid \alpha=(a_1,\ldots,a_r) \in \Phi \right\} \]
is a reduced root set.
%The \emph{simple roots} of $\Phi$ are the simple roots of the root system at 
%the object of the Cartan \scheme given by the definition of `root set'.

A \emph{root multiset} is a root set $\Phi$ together with a map $\Phi \rightarrow \NN$, $\alpha\mapsto m_\alpha$.
\end{definition}

\subsection{All parabolic restrictions of finite root systems}
By Theorem \nref{thm_EquivalenceArrangementCartanscheme} the root systems
are in $1:1$ correspondence with crystallographic arrangements.
Thus Theorem \ref{thm_ClassificationArrangments} is a complete classification of simply-connected Cartan \schemesn.

In the following we compute all restrictions to existing hyperplanes (parabolic restrictions, see
Definition \nref{def_ParabolicRestriction}) of all root systems in
terms of their arrangement:

\begin{theorem}\label{thm_AllRestrictions} In rank three, the \emph{only} 
crystallographic arrangements which are parabolic 
restrictions from higher dimensional crystallographic arrangements are those of 
the infinite series $A_3,B_3,C_3, \Ac_3^{1}(2),\Ac_3^{2}(2)$ and those sporadic 
arrangements labeled
\[ 1, 2, 3, 6, 7, 8, 9, 13, 14, 15, 20, 23 \]
in \cite{p-CH10}.
From rank four to rank eight, \emph{all} crystallographic arrangements except 
the sporadic reflection arrangements $E_6$, $E_7$, and $E_8$ are restrictions 
of higher dimensional reflection arrangements.\newline
Every crystallographic arrangement of rank greater than $8$ is restriction of a 
reflection arrangement of an infinite series (i.e.\ $A_n,B_n,C_n,D_n$).
\end{theorem}
\begin{proof}
The first two assertion are straightforward calculation performed by the computer using the classification, see Theorem \nref{thm_ClassificationArrangments}. The last assertion is proven as follows:
Every crystallographic arrangement in dimension greater than $8$ is either a 
reflection arrangement (Weyl arrangement) or an arrangement denoted 
$\Ac^{k}_{\ell}(2)$ in \cite[6]{OT}. But according to [OT92,Table 6.2], 
$\Ac^{k}_{\ell}(2)$ is the restriction of an arrangement $\Ac^{0}_{\ell'}$ for 
some $\ell'$ large enough, and $\Ac^{0}_{\ell'}$ is the arrangement of type 
$D_{\ell'}$. 
\end{proof}

\subsection{Folding restriction}

A second important source of examples is as follows: We say that a finite group $G$ acts on an arrangement $\Ac$ if $G$ acts on $V$ such that $g.\Ac=\Ac$ for any $g\in G$. Moreover we say $G$ acts on a crystallographic arrangement $(\Ac,R)$ if $g.R=R$ for all $g\in G$, where
$g$ acts on $V^*$ by $(g.\alpha)(v) = \alpha(g^{-1}(v))$ for $\alpha\in V^*$. \\
If there exists a chamber $K$ where $G$ permutes the simple roots, we call the action a \emph{permutation action} and it is equivalent to a permutation action of $G$ on the Dynkin diagram resp.\ Cartan matrix of $R$ in $K$.

\begin{definition}
		A \emph{folding restriction} of an arrangement $\Ac$ with an action of $G$ is the 
		restriction $\Ac^X$ to the subspace $X=V^G$ of fixed points. We similarly define a  
		folding restriction and a permutation restriction of a crystallographic arrangement 
		$(\Ac,R)$.
\end{definition}

In the sequel we will use the following notation. If $(\Ac,R)$ is crystallographic and $K$ is a chamber of $\Ac$, then
\[ R_+^K := \left\{ \alpha\in R \mid \alpha \in \sum_{\gamma\in \BDelta^K} \NN_0 \gamma \right\}. \]
Thus $R=R_+^K\dot\cup -R_+^K$ for every $K$. The sets $R_+^K$ should not be confused with the positive roots $R_+^a$ at an object of a Weyl groupoid. If $a$ is the object corresponding to the chamber $K$, then $R^a\subseteq \ZZ^r$ is the set of coordinate vectors with respect to $\BDelta^K$.

We show the following easy characterization which generalizes the approach in 
\cite{Ar62} for Lie type arrangements to classify real Lie algebras:

\begin{lemma}\label{lm_folding} Let $g$ be an involutive automorphism of the crystallographic arrangement $(\Ac,R)$. Then the folding restriction can be decomposed into two steps: First a parabolic restriction of $\Ac$ to $X_1\in L(\Ac)$ where $X_1$ is a suitable subspace invariant under $g$, then by a permutation restriction of $\Ac^{X_1}$ with respect to a suitable chamber $K_2$.
\end{lemma}
\begin{proof}
The proof yields $X_1,a_2$ and the explicit permutation action, and provides an efficient diagrammatic description of the possible actions of cyclic groups on crystallographic arrangements (comparable to Satake diagrams in \cite{Ar62}):

\begin{enumerate}[a)]
	\item Let $K$ be a chamber, such that $|g.R^K_+\cap -R^K_+|$ is minimal, which 	
		surely exists. Denote $\Delta=\BDelta^K$ the positive simple roots at $K$.
	\item If $\alpha\in\Delta$, then $g.\alpha\in R^K_+$ or $g.\alpha = -\alpha$: Otherwise, consider the chamber $K'$ adjacent to $K$ with $R^{K'}_+=(R^K_+\backslash\{\alpha\})\cup\{-\alpha\}$. If $g.\alpha=-\beta\in-R^{K}_+$ and $\beta\ne\alpha$, then
                $$|g.R^{K'}_+\cap -R^{K'}_+| = |g.R^{K}_+\cap -R^{K}_+|-1, $$
                which contradicts the assumed minimality. Note that we use here that $g$ is involutive, such that $g.\alpha=-\beta$ also imples $g^{-1}.\alpha=-\beta$, so no root in $R^K_+$ maps to $\alpha$. 
        \item Let $\Delta_1:=(-g^{-1}.\Delta) \cap \Delta$, thus the set of $\alpha\in\Delta$ with $g.\alpha=-\alpha$.
              Define $X_1:=\Delta_1^\perp\in L(\Ac)$. Since $g.\alpha=-\alpha$ for $\alpha\in\Delta_1$, $X_1$ is invariant under $g$.
              Further, if $v\in X = V^G$, then $g.v=v$. Hence $\alpha(v)=(-\alpha)(v)$ for all $	\alpha\in\Delta_1$, i.e.\ $X\subseteq X_1$.
        \item Consider the partition $\Delta:=\Delta_1\cup\Delta_2$, in particular $g.\Delta_2\subset R^K_+$. We claim that there is a permutation 
                $g(\_)$ of the set $\Delta_2$, such that $g.\alpha_i=\alpha_{g(i)}+\NN\Delta_1$. 
                Indeed, modulo $\RR\Delta_1$ the actions of $g,g^{-1}$ are inverse positive integer 
                matrices, hence permutation matrices.
	\item Consider the parabolic restriction $\Ac^{X_1}$ and let $K_2:= K \cap X_1$ be the 
		respective chamber, then by the above, $g$ acts as a permutation on the 
		crystallographic arrangement $\Ac^{X_1}$ and the full restriction to $X=V^G\subseteq X_1$ is 
		hence a permutation restriction.	   
\end{enumerate}
\end{proof}
\begin{remark}
	Even if $(\Ac^X,\bar{R})$ is not crystallographic, it could be crystallographic with 
	respect to a different choice of roots.
\end{remark}
	It is \emph{not} true that every parabolic restriction can be obtained from a folding 
	restriction for some suitable automorphism $g$:
\begin{lemma}
	For $J\subset I$, there exists an automorphism $g$, such that parabolic restriction to $X=J^
	\perp$ coincides with folding restriction by $g$, \emph{if and only if} the permutation automorphism 
	(i.e.\ diagram automorphism) 
	$f_J:=-w_J$ (on the parabolic subsystem generated by $J$ with $w_J$ the longest element), 
	together with the identity permutation on all simple roots $I\backslash J$ is a permuatation 
	automorphism for the entire arrangement.   
\end{lemma}
Take as counterexample $A_2\subset A_3$ or also $A_2\subset A_4$. The condition 
is however always fulfilled for parabolic restrictions to one hyperplane $H\in 
\Ac$.
\begin{proof}
Let $f$ be such an extension of $f_J=-w_J$ and consider the automorphism $g:=w_Jf$. On the subsystem $R_J^K$ it acts as $-\id$, while on the remaining positive roots it acts by adding terms in $R_{J,+}^K$ i.e.\ as the identity modulo $R_J^K$. We now apply Lemma \ref{lm_folding}: We see that for a simple root $\alpha_i\in\Delta^K$, by construction we have $g.\alpha\in R^K_+$ for $i\in I\backslash J$ or $g.\alpha_i = -\alpha_i$ for $i\in J$. Hence $\Delta_1=J$ and thus the first parabolic restriction $X_1=J^\perp$ is the parabolic restriction in question. Furthermore, since $g$ acts on the remaining simple roots by identity modulo $R_J^K$, the second permutation restriction is trivial.

Vice-versa, for any automorphism $g$ by Lemma \ref{lm_folding} we have a chamber $K$ and a set of simple roots $\Delta_1\subset \Delta$ with the property $g.\alpha_i = -\alpha_i$, and a permutation $\pi$ on the other roots $\Delta\backslash\Delta_1$ such that $g.\alpha\in\pi(\alpha)+R^K_{+,J}$. The folding restriction by $g$ is a parabolic restriction to $\Delta_1$ and then a permutation restriction.

Since it needs to coincide by assumption with parabolic restriction to $J$, we have $J\subset \Delta$, moreover $J=\Delta$ and $\sigma=\id$. We can consider $w_J^{-1}g$, which is a permutation automorphism $f$. On the subsystem $J$, $f$ is equal to $-w_J=:f_J$ and on the remaining roots the identity, as claimed.
\end{proof}

\begin{remark}Compare the condition of this Lemma to the condition fulfilled by the quotient root systems used for parabolic induction of cuspidal representations in Lusztig's character theory of finite Lie groups.
\end{remark}

Note that the permutation induces a permutation automorphism of root systems already on the localization to $\Delta_2$, in particular a diagram automorphism of the sub-Dynkin diagram. One may hence enumerate all possible automorphism by listing root systems, where for some object $a$, $R^a$ has a symmetric sub-diagram (and say color all nodes in $\Delta_1$ black) as done for Satake diagrams. Alternatively (and maybe more feasible for us) one may compute all parabolic restrictions as done below and then enumerate all full permutation automorphisms.     
A different way to classify root system automorphisms (even if $g$ is \emph{not involutive}) is
\begin{lemma}
For any automorphism $g$ of a crystallographic arrangement and any chamber $K$ we can write 
$g=wf$, where $w$ is the Weyl groupoid element defined by $w(K)=g(K)$ and $f$ is a permutation 
automorphism of the object $K$ (so in particular a diagram automorphism of the respective Dynkin 
diagram, possibly trivial). 
\end{lemma}
\begin{proof}
This is easily proven as for Lie algebras: The autmorphism $g$ has to send a chamber $K$ to 	
some chamber $K'$. By transitivity of the Weyl groupoid on the set of chambers we find a 
(unique) $w$ with $w(K)=K'$. Then both $w^{-1}g$ and $g^{-1}w$ leave $K$ invariant and act 
hence as mutually inverse, integral	positive matrix on $R^K_+$. They are hence again 
permutation matrices.
\end{proof} 

It is not as easy as for parabolic restriction to determine the resulting root system of a folding restriction. In the special case of permutation restriction, the new set of roots $\bar{R}$ consists of the orbits of $R$ under $G$: 

\begin{example}\label{exm_PermutationRestriction} In all of the following examples, the diagram automorphism inducing $g$ is pictured, as well as the resulting root system including multiplicities. Note that foldings for Lie type such as a)-c) are of course well-known.  
\begin{enumerate}[a)]
\item Let $R^a_+=\{(1,0),(0,1)\}$ be of Lie type $A_1\times A_1$ and $g$ the permutation transposing $\alpha_1\leftrightarrow \alpha_2$. A basis for $V$ is $\alpha_1^\vee,\alpha_2^\vee$, hence the invariant subspace is $X=v\RR:=(\alpha_1^\vee+\alpha_2^\vee)\RR$. The restriction of the arrangement is of course of type $A_1$ (two of the four chambers are intersected). The restrictions of the roots $\alpha_1,\alpha_2\in R$ to $(\Ac^X,\bar{R})$ are both $\bar{\alpha}_1:=(v\mapsto 1)$ which can be expressed in terms of the orbit $G\alpha_1=\{\alpha_1,\alpha_2\}$ as follows
 $$\bar{\alpha}_1=\frac{\alpha_1+\alpha_2}{2}=\alpha_{\{1,2\}}
\qquad \mbox{where } \alpha_{G\alpha}:=\frac{1}{|G|}\sum_{g\in G} g.\alpha$$
(this also justifies our initial choice of scaling $v=\bar{\alpha}_1^\vee$).  
Hence $\bar{R}^a_+$ is reduced of Lie type $A_1$ with root 
$\bar{\alpha}_1=\alpha_{\{1,2\}}$. The root $\bar{\alpha}_1$ 
has \emph{multiplicity} $2$ since two roots $\alpha_1,\alpha_2$ (resp.\ hyperplanes $H_1,H_2$) restrict to the root $\bar{\alpha}_1$ (resp.\ hyperplane); see Section \ref{sec_ParabolicRestriction}.
\begin{center}
\includegraphics[scale=.18]{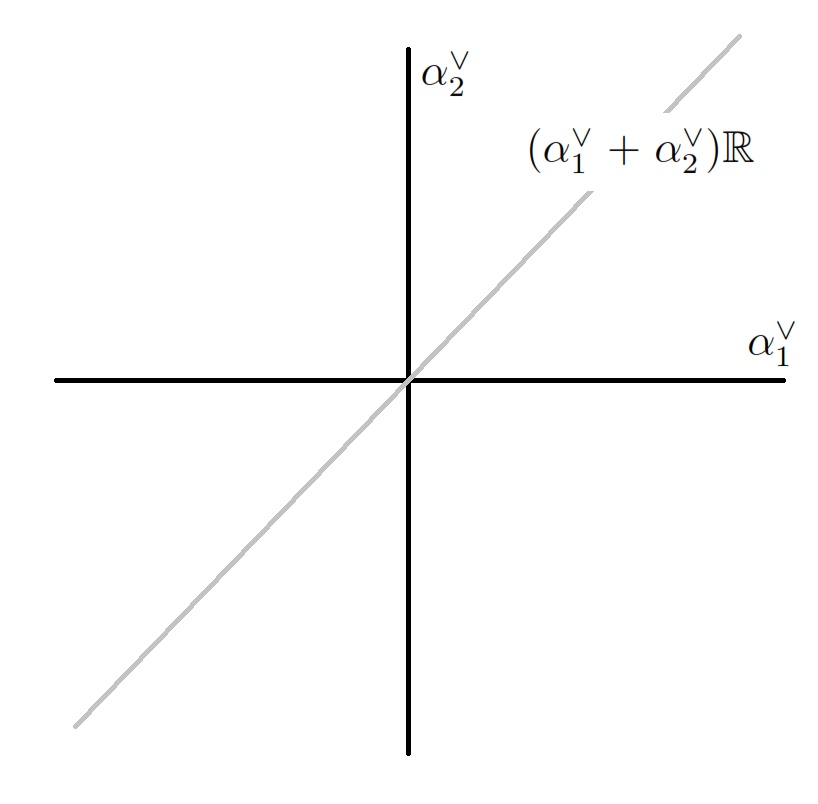}
\hspace{3cm}
\includegraphics[scale=.25]{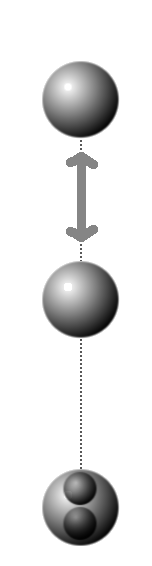}
\end{center}
\item Let $R^a_+=\{(1,0,0,0),(0,1,0,0),(0,0,1,0),(0,0,0,1),(1,0,1,0),(0,1,0,1)\}$ be of Lie type $A_2\times A_2$ and $g$ the permutation transposing $\alpha_1\leftrightarrow \alpha_2,\alpha_3\leftrightarrow \alpha_4$. Then the restriction is $\bar{R}^a_+=\{(1,0),(0,1), (1,1)\}$ with two simple roots and one non-simple root each corresponding to an orbit of length $2$:
\begin{align*}
\bar{\alpha}_1 &=\alpha_{\{1,2\}}\\
\bar{\alpha}_2 &=\alpha_{\{3,4\}}\\
\bar{\alpha}_1+\bar{\alpha}_2&=\frac{\alpha_1+\alpha_2+\alpha_3+\alpha_4}{2}
=\alpha_{\{13,24\}}
\end{align*}
The restriction is hence a reduced root system of Lie type $A_2$ and all roots have multiplicity $2$. 
\begin{center}
 \includegraphics[scale=.25]{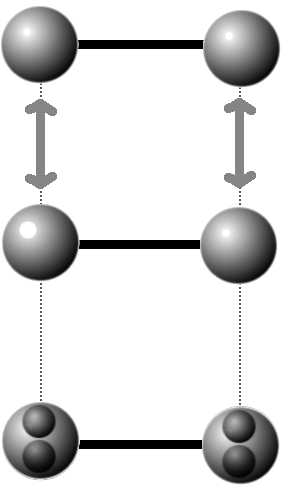}
\end{center}
\item Let $R^a_+=\{(1,0,0),(0,1,0),(0,0,1),(1,1,0),(0,1,1),(1,1,1)\}$ be of Lie type $A_3$ and $g$ the permutation transposing $\alpha_1\leftrightarrow \alpha_3$. Then the restriction is $\bar{R}^a_+=\{(1,0),(0,1), (1,1),(2,1)\}$ with the following roots:
\begin{align*}
	\bar{\alpha}_1 &= \alpha_{\{1,3\}} \\
	\bar{\alpha}_2 &= \alpha_{2} \\
	\bar{\alpha}_1+\bar{\alpha}_2 &= \frac{\alpha_1}{2}+\alpha_2+\frac{\alpha_3}{2}=\alpha_{		\{12,23\}}\\
	2\bar{\alpha}_1+\bar{\alpha}_2 &= \alpha_1+\alpha_2+\alpha_3=\alpha_{123}\\
\end{align*}
The restriction is hence reduced of Lie type $B_2$ with short roots of 
multiplicity $2$ and long roots of multiplicity $1$.
\begin{center}
\includegraphics[scale=.6]{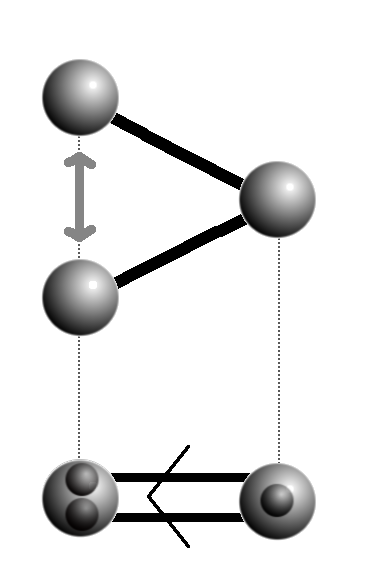}
\end{center} 
\item Let $R^a_+=\{(1,0,0),(0,1,0),(0,0,1),(1,1,0),(1,0,1),(0,1,1),(1,1,1)\}$, which is \emph{not} of Lie type, let $g$ be the permutation transposing $\alpha_1\leftrightarrow \alpha_3$. Then the restriction is $\bar{R}^a_+=\{(1,0),(0,1), (1,1),(2,0),(2,1)\}$ and hence of Lie type $B_2$ and multiplicities $1,2$ as in the previous case, but this time non-reduced, i.e.
$$[M_{\bar{\alpha_1}},M_{\bar{\alpha_1}}]=M_{2\bar{\alpha_1}}\neq 0$$
\begin{center}
\includegraphics[scale=.6]{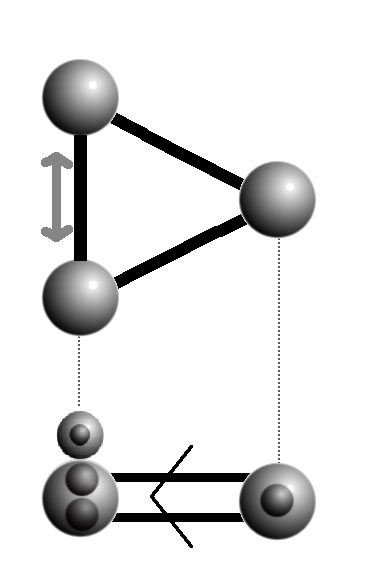}
\end{center}
A similar effect appears already when folding a so-called loop $A_2$, see \cite{Len14a}. 
\end{enumerate}
\end{example}

We finally give examples for general folding restrictions.
\begin{example} Similarly to Satake diagram (first example), according to the 
proof of Lemma \nref{lm_folding}, we draw the permutation of the simple roots in 
$\Delta_2$ and denote the simple roots in $\Delta_1$ by blackend dots.

 Let $R^K_+=\{(1,0,0),(0,1,0),(0,0,1),(1,1,0),(0,1,1),(1,1,1)\}$ be of Lie type 
	$A_3$ and $g$ the automorphism
	$$g:=\begin{pmatrix} 0 & 0 & 1 \\ 1 & -1 & 1 \\ 1 & 0 & 0\end{pmatrix}.$$
	We have $g.R^K_+\cap -R^K_+=\{-\alpha_2\}$ which is minimal (so $K$ is already chosen 
	suitably) and thus $\Delta_1=\{\alpha_2\}$ and $\Delta_2=\{\alpha_1,\alpha_3\}$. 
	Moreover, modulo $\Delta_1$ we have a 
	permutation $g:\alpha_1\leftrightarrow \alpha_2$. The restriction to $X_1=\alpha_2^\perp$ 
	is of type $A_2$ with multiplicities $2$ for the two simple roots and $1$ for the non-simple root.
	Note that while the root system is of Lie type, the multiplicities are not 
	invariant under all reflections and we thus get a Weyl groupoid covering the Weyl group $
	A_3$ (similar to a Lie superalgebra). As a second step, the permutation restriction by 
	the transposition $g$ yields a non-reduced root system of Lie type $A_1$ where $\bar{
	\alpha}_1$ has multiplicity $4$ and $2\bar{\alpha}_1$ has multiplicity $1$.
	\begin{center}
	 \includegraphics[scale=.6]{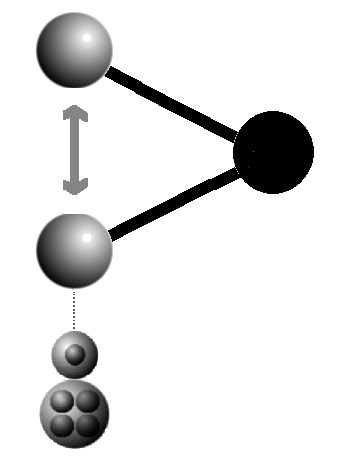}
	\end{center}
\end{example}

\section{Arrangements and Nichols algebras}\label{sec_arrnich}

\subsection{Nichols algebras}

Let $\CC$ be a braided category, then there is a straightforward notion of a Hopf algebra in 
$\CC$. For this and especially the notion of Yetter-Drinfel'd modules in braided categories see e.g.\ \cite{BLS14}. To simplify the discussion we assume in this article 
$\CC=\YD{h}$ where $h$ is a complex Hopf algebra.

A particularly well-known example class are the diagonal Nichols algebras. In Section \ref{sec_examples} the reader can find a table comparing the notation for general Nichols algebras 
with the notation for quantum groups and general diagonal Nichols algebras in \cite{Lusz94} and \cite{Heck07}.

\begin{definition}
Let $M=M_1\oplus \cdots \oplus M_n$ be a semisimple object in $\CC$ with simple summands $M_i$. We call $n$ the \emph{rank} of $M$. The \emph{Nichols algebra} $\B(M)$ is an $\NN$-graded Hopf algebra in the category $\CC$ with $\B(M)_0=1\C$ the identity object in $\CC$ and $\B(M)_1=M$ primitive elements $\Delta(m)=1\otimes m+m\otimes 1$ which generate $\B(M)$ as an algebra.\\ The Nichols algebra $\B(M)$ can be defined by the following equivalent properties:
\begin{itemize}
\item Every $\NN$-graded Hopf algebra $B'$ with the above properties admits a unique graded Hopf algebra surjection $B'\to \B(M)$ which is the identity on $M$. That is, $\B(M)$ is the quotient of the tensor algebra $T(M)$ by the largest Hopf ideal $\mathfrak{J}$ in degree $\geq 2$.
\item The pairing $M\otimes M^*\to \C$ induces a Hopf pairing $\B(M)\otimes \B(M^*)\to \C$ which is nondegenerate. That is, $\B(M)$ is the quotient of the tensor algebra $T(M)$ by the radical of the Hopf pairing $T(M)\otimes T(M^*)\to \C$.
\item More explicitly, $\B(M)$ is the quotient of $T(M)$ in each degree $k$ by the kernel of the \emph{quantum symmetrizer}, which is defined in terms of the braid group action $M^{\otimes k}\to M^{\otimes k}$.
\end{itemize} 
While the last characterization enables one in principle to compute $\B(M)$ in each degree, it is extremely difficult to find generators and relations for $\B(M)$ or even determine for a given $M$ if it is finite-dimensional. 
\end{definition}

Since $\B(M)$ is an $\NN$-graded algebra with finite-dimensional degree layers, we may consider the \emph{Hilbert series}, a formal power series
$$\mathcal{H}(t):=\sum_{k\geq 0}\dim\left(\B(M)_k\right) t^k.$$ 
In particular, if the dimension is finite then $\dim(\B(M))=\mathcal{H}(1)$. We frequently use the symbol $(n)_t:=1+t+t^2+\cdots+t^{n-1}$.\\

The first examples serves to fix notation:

\begin{example}[Diagonal case]
	Let $h=\C[\Gamma]$ and $\Gamma$ be an abelian group. Then the braided category $\CC=\YD{h}$ is 
	semisimple. The finite-dimensional simple objects are $1$-dimensional vector spaces $M_i=x_i\C$ 	
	together with a group element $g_i\in \Gamma$ and a linear character $\chi_i:G\to \C^\times$. 
	The braiding is given by 
	$$M_i\otimes M_j\to M_j\otimes M_i,$$
	$$x_i\otimes x_j\mapsto q_{ij}\;x_j\otimes x_i,\quad q_{ij}:=\chi_j(g_i).$$
	We call such a braiding \emph{diagonal} and $q_{ij}$ the \emph{braiding matrix}. From the third 
	characterization we see	that the Nichols algebra $\B(M)$ depends only on the braiding matrix 
	$(q_{ij})_{i,j}$ of $M$. Conversely, every diagonal braiding can be realized as Yetter-Drinfel'd 
	module, say over $\Gamma=\Z^n$.
\end{example}

\begin{example}
Let $q\in \C^\times$. Let $M=M_1$ be the one-dimensional 
complex braided vector space with 
basis $x_1$ and braiding matrix $(q_{11})=(q)$. The Nichols algebra is  
$$\B(M)\cong \C[x_1]/(x_1^\ell)$$
if $q$ is a primitive $\ell$-th root of unity, and $\B(M)\cong\C[x_1]$ if $q=1$ or not a root of unity. In the first case, the Hilbert series is $(\ell)_t$, in the second case $\frac{1}{1-t}$.
\end{example}

The next example exhibits the role of Nichols algebras 
as quantum Borel parts.

\begin{example}
Let $q$ be a primitive $\ell$-th root of unity and assume for simplification $2,3\nmid\ell$ and $\ell>3$. Let $\mathfrak{g}$ be a complex finite-dimensional semisimple Lie algebra of rank $n$, let $\alpha_1,\ldots\alpha_n$ be a set of simple roots, $R_+$ the set of positive roots, and let $(,)$ be the Killing form, normalized to $(\alpha_i,\alpha_i)=2$ for short simple roots $\alpha_i$.\\
Consider the Yetter-Drinfel'd modules $M=M_1\oplus\cdots\oplus M_n$ over the abelian group $\Gamma=\Z_\ell^n$ generated by $g_i$ and $M_i=E_i\C$ with group element $g_i$ and character $\chi_i(g_j):=q^{(\alpha_i,\alpha_j)}$.\\ 
Then $\B(M)$ is finite-dimensional with Hilbert series $\prod_{\alpha\in R_+}(\ell)_{t^{deg(\alpha)}}$, thus of dimension $\ell^{|R_+|}$. In fact $\B(M)\cong u_q(\mathfrak{g})^+$ is as an algebra the positive part of the small quantum group.
\end{example}

\subsection{The Weyl groupoid of a Nichols algebra}\label{sec_RootSystemsNicholsAlgebras}

The structure theory of Nichols algebras is dominated by the structure of the Weyl groupoid, which generalizes the role of reflection operators in quantum groups as introducted by Lusztig \cite{Lusz90}, and allows to define a root system for the Nichols algebra. For $M,q_{ij}$ diagonally braided (see above) Heckenberger has introduced in a series of papers reflections \cite{Heck07} and an arithmetic root system in terms of the bicharacter induced by $q_{ij}$ and finally classified all finite-dimensional Nichols algebras in terms of their root systems in \cite{Heck09}. More generally in \cite{AHS10}\cite{HS10} Andruskiewitsch, Heckenberger, and Schneider have introduced a Weyl groupoid and root system for arbitrary semisimple $M\in\YD{h}$. We shall sketch their approach in the form discussed in 
\cite{HS13} or \cite{BLS14}:\\

Let $M=M_1\oplus \cdots \oplus M_n$ be a finite-dimensional semisimple object 
in $\CC=\YD{h}$ and assume for simplicity that $\B(M)$ is already 
finite-dimensional. For any $i$ the projection $M\to M_i$ induces a projection 
$\pi_\B(M)\to \B(M_i)$. By the \emph{Radford projection theorem} we may write 
$$\B(M)\cong K_i \rtimes \B(M_i),\qquad K_i:=\B(M)^{coin\;\B(M_i)}:=\{h\mid 
(\id\otimes \pi_i)\Delta(h)=h\otimes 1\}$$
where the space of coinvariants $K_i$ with respect to $\pi_i$ is a Hopf algebra 
in the braided category $\YD{\B(M_i)}(\CC)$ and $\rtimes$ is the \emph{Radford 
biproduct}.\\  

By the second characterization of Nichols algebras we have a non-degenerate 
Hopf pairing $\B(M)\otimes \B(M^*)\to \C$, in particular 
$\dim(\B(M))=\dim(\B(M^*))$. One can show that this Hopf pairing induces a 
category equivalence $\Omega:\YD{\B(M)}\to \YD{\B(M^*)}$. Now we may turn 
$\Omega(K_i)\in\YD{\B(M^*)}$ again into a Hopf algebra in $\CC$, the 
\emph{reflection}:
$$r_i(\B(M)):=\Omega(K_i)\rtimes \B(M^*). $$
Reflection does clearly not change the dimension  
$\dim(\B(M))=\dim(r_i(\B(M)))$. Otherwise $\B(M),r_i(\B(M))$ can be quite 
different, although one can prove $\YD{\B(M)}\cong 
\YD{r_i(\B(M)}$. The reflection operation including the mentioned properties 
are neither restricted to Nichols algebras nor to the category $\CC=\YD{h}$ and 
were dubbed \emph{partial dualization} in \cite{BLS14}.\\

For a Nichols algebra $\B(M)$, Andruskiewitsch, Heckenberger and Schneider 
describe this operation in much more detail: In particular, $r_i(\B(M))\cong 
\B(R_i(M))$ is again a Nichols algebra for some explicit $R_i(M)\in \CC$. We 
summarize some results of \cite{AHS10}, \cite{HS10}, and \cite{HS13}:

\begin{theorem}\label{thm_NicholsAlgebraRootSystem}
Let $h$ be a complex Hopf algebra. Let $M_i$ be a finite
collection of simple $h$-Yetter-Drinfel'd modules. Consider 
$M:=\bigoplus_{i=1}^n M_i\in\YD{h}$ and assume that the associated
Nichols algebra $H:=\B(M)$ is finite-dimensional. 
Then the following assertions hold:
\begin{itemize}
\item By construction, the Nichols algebras 
$\B(M),r_i(\B(M))$ have the same dimension as complex vector spaces.
\item
For $i\in I$, denote by $\hat M_i$ the braided subspace
$$\hat M_i=M_1\oplus \ldots\oplus M_{i-1}\oplus M_{i+1}\oplus\ldots \oplus M_n$$
of $M$.
Denote by $\ad_{\B(M_i)}(\hat M_i)$ the braided vector space
obtained as the image of $\hat M_i\subset\B(M)$ under the adjoint action
of the Hopf subalgebra $\B(M_i)\subset\B(M)$. Then,
there is a unique isomorphism \cite[Prop. 8.6]{HS13} 
$$K_i\cong \B(\ad_{\B(M_{i})}(\hat M_i))$$
of Hopf algebras in the braided category 
$\YD{\B(M_i)}\left(\YD{h}\right)$
which is the identity on $\ad_{\B(M_i)}(\hat M_i)$.
\item  Define, with the usual convention for the
sign,
$$-c_{ij}:=\max\{m\;|\;\ad_{M_i}^m(M_j)\neq0\} \qquad c_{ii}=2.$$
Fix $i\in I$ and denote for $j\neq i$
$$V_j:=\ad_{M_i}^{-c_{ij}}\left({M_j}\right)\subset\B(M)\,\,.$$
The braided vector space 
$$R_i(M)=V_1\oplus\cdots M_i^*\cdots \oplus V_n
\in\YD{h}$$ 
is called the the $i$-th reflection of the braided vector
space $M$.
Then there is a unique isomorphism \cite[Thm.\ 8.9]{HS13}
of Hopf algebras in $\YD{h}$
$$r_i(\B(M_1\oplus\cdots \oplus M_n))\cong \B(V_1\oplus\cdots
M_i^*\cdots \oplus V_n)$$
which is the identity on $M$.
\item With the same definition for $c_{ij}$ for $i\neq j$
and $c_{ii}:=2$, the matrix $(c_{ij})_{i,j=1,\ldots n}$ is a
generalized Cartan matrix \cite[Thm. 3.12]{AHS10}. 
Moreover, one has $r_i^2(\B(M))\cong \B(M)$,
and the Cartan matrices coincide, $c_{ij}^M=c_{ij}^{r_i(M)}$. 
One obtains a Cartan \scheme where each object $a\in A$ corresponds to some reflection $\B(M^a)$. 
\item The maps $r_i$ give rise to a Weyl groupoid: The objects are the different isomorphism classes of tuples $(M_1,\ldots,M_n)$ and the formal morphisms $M\to r_i(M)$ are generated by reflections $\Z^n\to \Z^n$ with respect to $c_{ij}^M$.  
\item Let $R^a_+$ the set of positive roots, and choose a reduced expression for the longest element in the Weyl groupoid, then there is an isomorphism of $\NN$-graded objects in $\CC$ (not algebras), or PBW-basis:
$$\B(M^a)=\bigotimes_{\alpha\in R^a_+} M_\alpha.$$ 
Here $M_\alpha$ are certain simple object in $\CC$, namely $=M^a_i$ if $H_\alpha$ is adjacent to the object $a$. For details, we refer to \cite[Sect.\ 3.5]{AHS10}
and \cite[Sect.\ 5]{HS10}.
\end{itemize}
\end{theorem}

Note that a crucial point is that the algebras $\B(M),r_i(\B(M))$ can be quite different and in particular their Cartan matrix $(c_{ij})_{ij}$ may be different -- in contrast to quantum groups, where all reflections are isomorphic as algebras. This is why we obtain a Weyl groupoid instead of a Weyl group, a Cartan \scheme instead of a single Cartan matrix, and more root systems than for semisimple Lie algebras.

\begin{theorem}[Diagonal case, \cite{Heck09}]
For a diagonally braided vector space $M,q_{ij}$  
Heckenberger had previously described the reflection $r_i(M), q_{ij}'$ as the 
explicit base transformation of the bicharacter induced by the reflection 
$\alpha_j\mapsto \alpha_j-c_{ij}\alpha_i,\;\alpha_i\mapsto -\alpha_i$, where the 
Cartan matrix $(c_{ij})_{ij}$ associated to $(M,(q_{ij})_{i,j})$ can be 
calculated by
$$c_{ii}=2,\qquad c_{ij}=-\min\{m\mid (m+1)_{q_{ii}}=0 \:\text{ or }\: 
q_{ii}^m q_{ij}q_{ji}=1\},
\quad i\neq j.$$

  Any finite-dimensional complex Nichols algebra in the
  categoy of Yetter-Drinfel'd  modules over an abelian group $\YD{\Gamma}$
  appears in Heckenberger's list \cite{Heck09}.
\end{theorem}

The respective root systems are many, but not all possible root systems and Weyl groupoids as classified by Heckenberger and the first author, see Section \nref{sec_ClassificationWeylGroupoid}. For instance, there are infinitely many Weyl groupoids of rank $2$ and many more exceptional Weyl groupoids in rank $3$.\\

The following class of examples of type $B_2$ contains as special case the Borel part $u_q(B_2)^+$ when the braiding matrix $(q_{ij})_{ij}$ is chosen symmetric.

\begin{example}
Let $q,q^4\neq 1$ be an $\ell$-th root of unity and $M_{\alpha_1},M_{\alpha_2}$ 
be $1$-dimensional Yetter-Drinfel'd modules (say over $\Gamma=\ZZ^2$) such that 
the diagonal braiding of $M^a:=M_{\alpha_1}\oplus M_{\alpha_2}$ fulfills
$$q_{\alpha_1\alpha_1}=q^2,\qquad q_{\alpha_i\alpha_j}q_{\alpha_j\alpha_i}=q^{-4},
\qquad q_{\alpha_2\alpha_2}=q^4.$$
The previous formula yields the Cartan matrix and Dynkin diagram $M^a$:
\begin{center}
$(c_{ij}^a)_{i,j}=\begin{pmatrix} 2 & -1 \\ -2 & 2 \end{pmatrix}$ \quad
\includegraphics[scale=.6]{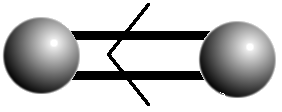}
\end{center}
The reflection on the short root $\alpha_1$ given by this Cartan matrix gives for $r_1(M)$
$\alpha_1'=-\alpha_1$ and $\alpha_2'=\alpha_2+2\alpha_1$. For the braiding matrix of 
$r_1(M)$ we thus get by base transformation of the bicharacter:
\begin{align*}
	q_{\alpha_1'\alpha_1'}&=q_{\alpha_1\alpha_1}=q^2\\
	q_{\alpha_2'\alpha_2'}&=q_{\alpha_2\alpha_2}
		q_{\alpha_1\alpha_2}^2q_{\alpha_2\alpha_1}^2q_{\alpha_1\alpha_1}^4=q^{4-8+8}=q^4\\
	q_{\alpha_1'\alpha_2'}&=q_{\alpha_1\alpha_2}^{-1}q_{\alpha_1\alpha_1}^{-2}\\
	q_{\alpha_2'\alpha_1'}&=q_{\alpha_1\alpha_2}^{-1}q_{\alpha_1\alpha_1}^{-2}
\end{align*}
We observe that in general the Yetter-Drinfel'd module $r_1(M)$ given by $q_{\alpha_i'\alpha_j'}$ is \emph{not} isomorphic to $M$ (for the special case $q_{ij}=q_{ji}$ appearing in $u_q(B_2)^+$ it is!). However we easily check
$$q_{\alpha_1\alpha_2}'q_{\alpha_2\alpha_1}'
=q_{\alpha_1\alpha_2}^{-1}q_{\alpha_2\alpha_1}^{-1}q_{\alpha_1\alpha_1}^{-4}
=q^{4-8}=q_{\alpha_i\alpha_j}q_{\alpha_j\alpha_i},$$
so $r_1(M)=M_{\alpha_1'}\oplus M_{\alpha_2'}$ has the same Cartan matrix as $M$ (this can not be deduced merely from the Dynkin diagram). The same can be checked for $r_2(M)$. We thus obtain a Cartan \scheme with a single Cartan matrix for all reflections $M^a$. The arrangement and set of positive roots is of Lie type $B_2$:
$$R^a_+=\{\alpha_1,\alpha_2,\alpha_{12},\alpha_{112}\}=\{(1,0),(0,1),(1,1),(2,1)\}$$
with notation $\alpha_{12}:=\alpha_1+\alpha_2$, $\alpha_{112}:=2\alpha_1+\alpha_1$. The self-braidings are 
$$q_{\alpha_1\alpha_1}=q_{\alpha_{12}\alpha_{12}}=q^2, 
\qquad q_{\alpha_2\alpha_2}=q_{\alpha_{112}\alpha_{112}}=q^4.$$ 
Hence the PBW-basis Theorem gives an isomorphism of $\NN$-graded $\Gamma$-Yetter-Drinfel'd modules:
\begin{align*}
\B(M^a)&\cong \bigotimes_{\alpha\in R^a_+} M_\alpha\\
&=\C[t_{1}]/(t_{1}^{\ord(q^2)})\otimes\C[t_{2}]/(t_{2}^{\ord(q^4)})\otimes
\C[t_{12}]/(t_{12}^{\ord(q^2)})\otimes\C[t_{112}]/(t_{112}^{\ord(q^4)}).
\end{align*}
For $\ell$ odd, $q^2,q^4$ both have order $\ell$, hence the Hilbert series is $(\ell)_t^2(\ell)_{t^2}(\ell)_{t^3}$ and in particular the dimension is $\ell^4$.
\end{example}

We also give an example which is not of Lie type and which we will use in the following. It is a finite-dimensional Nichols algebra $\B(M)$ of diagonal type of rank $3$ appearing in \cite[row $11$]{Heck09}. The rows $9,10$ define the same arrangement, but different roots of unity involved and hence  more types of objects with the same Cartan matrix. The arrangement associated to this Nichols algebra as well as all restrictions are calculated in Section \nref{sec_ExampleSimplicialComplex}.

\begin{example}\label{exm_CycleNicholsAlgebra}
Let $M_{\alpha_1},M_{\alpha_2},M_{\alpha_3}$ be $1$-dimensional 
Yetter-Drinfel'd modules (say over $\ZZ^3$) such that the diagonal 
braiding of $M^a:=M_{\alpha_1}\oplus M_{\alpha_2}\oplus M_{\alpha_3}$ fulfills 
$$q_{\alpha_i\alpha_i}=-1,\qquad q_{\alpha_i\alpha_j}q_{\alpha_j\alpha_i}=\zeta,$$ 
with $i\neq j$ and $\zeta$ a primitive third root of unity. The associated crystallographic arrangement has $7$ roots. It is called $\mathcal{A}_3^1(2)$ and the first member of a series.
$$R^a_+=\{(1,0,0),(0,1,0),(0,0,1),(1,1,0),(1,0,1),(0,1,1),(1,1,1)\}.$$
From this we easily see the Dynkin diagram/Cartan matrix is a simply-laced triangle and not of Lie type:
\begin{center}
  \includegraphics[scale=.6]{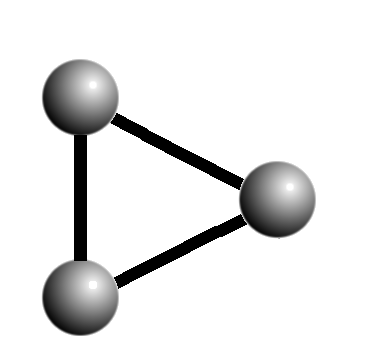}
\end{center}
The self-braidings of these roots are
$$q_{\alpha_i\alpha_i}=-1,\qquad q_{\alpha_{ij},\alpha_{ij}}=\zeta,
\qquad q_{\alpha_{123}}=-1.$$
Hence the Nichols algebra $\B(M^a),\;M^a=M_{\alpha_1}\oplus M_{\alpha_2}\oplus M_{\alpha_3}$ has Hilbert series $(2)_{t}^3(3)_{t^2}^3(2)_{t^3}$ and dimension $432$.

We calculate the reflection to an object $a'$ with respect to $\alpha_2$: The simple roots at $a'$ are then 
$\alpha'_1:=\alpha_{12}$, $\alpha'_2:=-\alpha_2$, $\alpha'_3:=\alpha_{23}$ and the braiding matrix fulfills
$$q_{\alpha'_1\alpha'_1}=q_{\alpha'_3\alpha'_3}=\zeta,\qquad q_{\alpha'_2\alpha'_2}=-1,$$ 
$$q_{\alpha'_1\alpha'_2}q_{\alpha'_2\alpha'_1}=
q_{\alpha'_2\alpha'_3}q_{\alpha'_3\alpha'_2}=\zeta^{-1},
\qquad q_{\alpha'_1\alpha'_3}q_{\alpha'_3\alpha'_1}=1.$$
The respective roots in the basis $\alpha'_1,\alpha'_2,\alpha'_3$ can easily be either again calculated or directly read off from transforming $R^a_+$:
$$R^{a'}_+=\{(1,0,0),(0,1,0),(0,0,1),(1,1,0),(0,1,1),(1,2,1)\}$$
which has Dynkin diagram/Cartan matrix $A_2$, but is not of Lie type, since it has a different root system. 
\begin{center}
\includegraphics[scale=.6]{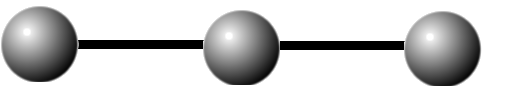}
\end{center}

The Hilbert series is now 
$(2)_{t}(3)_{t}^2(2)_{t^2}^2(2)_{t^3}(3)_{t^4}$ but of course the dimension is still $432$.
\end{example}

We finally discuss nondiagonal Nichols algebras over nonabelian groups.
\begin{definition}
  Let $M$ be a Yetter-Drinfel'd module over the group ring of $\C[G]$ of a 
nonabelian group $G$. Such an $M$ decomposes into irreducible summands 
$M_{\alpha_i}=M([g_i],\chi_i)$ associated to a conjugacy class $[g_i]\subset G$ 
and an irreducible representation $\chi$ of the centralizer of $g_i$ (see e.g.\ 
\cite{HS10}). In particular $M=\bigoplus_{i} M_{\alpha_i}$ is a nondiagonally 
braided vector space and we study the Nichols algebra $\B(M)$.
\end{definition}
The first  finite-dimensional 
Nichols algebra of this type has been constructed first for the dihedral group 
$G=\D_4$ in \cite{MS00} in the context 
of Coxeter groups. Later, in \cite{Len14a} it has been constructed by the 
second author as smallest example of a new family of large-rank 
finite-dimensional indecomposable Nichols algebras, introducing the folding 
construction of non-diagonal Nichols algebras from diagonal ones (in this case 
$A_2\times A_2$ with $q=-1$). The reader may compare the root system below with 
the root system in \nref{exm_PermutationRestriction} b). A similar construction 
for c) returns a decomposable Nichols algebra of type $B_2$.

\begin{center}
	\includegraphics[scale=.25]{DynkinA2A2A2.png}
	\includegraphics[scale=.65]{DynkinA3B2.png}
\end{center} 

\begin{example}
	Let $\Gamma=\D_4=\langle a,b\rangle/(a^4=b^2=1,ba=a^{-1}b)$ then one can define $2$-dimensional 
	simple Yetter-Drinfel'd modules $M_{\alpha_1},M_{\alpha_2}$ associated to the conjugacy classes 
	$[ab],[b]$. In fact, the Nichols algebra $\B(M)$ for $M=M_{\alpha_1}\oplus M_{\alpha_2}$ has a 
	root system of type $A_2$ and $M_{\alpha_{12}}$ is also $2$-dimensional and associated to the 
	conjugacy class	$[a]$. Altogether this Nichols algebra has Hilbert series $(2)_t^4(2)_{t^2}^2$ 
	and thus dimension $2^6$, indeed just like $u_q(A_2\times A_2)^+$ with $q=i$.
\end{example}

Note that in \cite{Heck09} all finite-dimensional complex Nichols algebras over abelian groups (equivalently with diagonal braiding) have been classified, and very recently in \cite{HV14} all finite-dimensional complex indecomposable Nichols algebras over nonabelian groups. For later use we note that only a subset of all finite Weyl groupoids (Theorem \nref{thm_ClassificationArrangments}) appear in these cases: 

\begin{corollary}\label{cor_NicholsAlgebraExistingRootsystems}
Using the labeling introduced in \cite{p-CH10},
the following crystallographic arrangements of rank three come from Weyl 
groupoids of finite-dimensional Nichols algebras of diagonal type\footnote{For 
the readers convenience: The super Lie algebras (numbered 4-10 in 
Heckenberger's list) have root systems $A_3,B_3,\Ac_3^{1}(2),\Ac_3^{2}(2)$ with 
respectively $6,9,7,8$ roots, as well as the sporadic root system labeled $6$ 
with $13$ roots. The rather alien Nichols algebras with numbers $16,17$ in 
Heckenberger's list have the sporadic root systems labeled $1,2$ with $10$ resp.\ 
$11$ roots.},
\[ \text{types } A_3,B_3,C_3,\Ac_3^{1}(2),\Ac_3^{2}(2),\quad \text{and the sporadic ones labeled } 1,2,3,6,8,9, \]
and the following come from complex finite-dimensional indecomposable Nichols algebras over nonabelian groups,
\[\text{types } A_3,C_3 ,\quad \text{and the sporadic one labeled }9. \]
Note that by Theorem \nref{thm_AllRestrictions}, all of these are (parabolic) restrictions of Weyl arrangements.
\end{corollary} 

\begin{question}
As already mentioned in the previous example, all such Nichols algebras of rank $<4$ over nonabelian groups had been uniformly constructed in \cite{Len14a} by \emph{folding}; they have root systems of type $A_n,C_n,E_6,E_7,E_8,F_4$. The resulting root system was calculated by hand and can now be understood as a restriction. It would be interesting to give a larger class of constructions:
\begin{itemize}
\item By considering foldings of Nichols algebras of non-Lie type, such as Example 
\nref{exm_PermutationRestriction} d).
\item By considering Satake-type foldings with nontrivial $\Delta_1$, which tends to drastically decrease the rank. This leads a-priori to Nichols algebras in other braided categories. 
\item In the affine setting, compare with the quantum affine algebra foldings in \cite{Len14b}.
\end{itemize}
In particular one would like to construct via some generalized folding the remaining exceptional cases of low-rank $2,3$ classified in \cite{HV14} and known (and possibly new) cases of rank $1$.
\end{question}

\begin{example}\label{exm_newNicholsAlgebra}
	In rank 3, apart from $A_3,C_3$, there is only one finite dimensional exceptional Nichols algebra over a 
	nonabelian group (defined for any field with a third 
	root of unity or any field characteristic $3$). It is \emph{not} of 
	Lie type, but has the root system rank 3 no.\ 9 with 13 roots and root space dimensions (in the first object $a$ with $B_3$ 
	Cartan matrix, see \cite[Lm.\ 8.8]{HV14}):
	
	\begin{center} 
		\includegraphics[scale=.6]{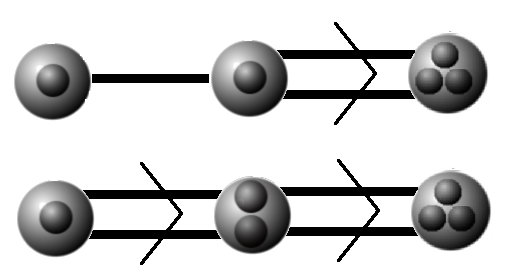}
	\end{center}
	\begin{center}
	\begin{tabular}{cc}
	$\alpha \in R^a_+$ & $\dim(M_\alpha)$\\
	\hline
	$1,2, 12,12^23^4,12^33^4, 1^22^33^4$ & 1 \\
	$23^2,123^2,12^23^2$ & 2\\
	$3,23,123,12^23^3$ & 3\\ 
	\end{tabular}
	\end{center}
	The rank 2 parabolic generated by $\alpha_2,\alpha_3$ is an exceptional Nichols algebra over a 
	nonabelian group and has a standard root system of type $B_2$. The three-dimensional rank 1 
	parabolics are the Nichols algebras of dimension $12$ over the nonabelian group $\S_3$, which 
	was the first known Nichols algebra over a nonabelian group \cite{MS00}.\\
	
	Ideally one would want to construct also these Nichols algebras via restriction, and we start 
	our search with purely parabolic restrictions (no diagram automorphism). \emph{Indeed}
	we find, that a restriction of the sporadic root system of rank 4 no.~7 with $25$ roots (and the parabolics rank 3 no.~2 and $B_2$) have
	precisely the right multiplicities. Moreover, they can be realized (in characteristic $
	0$ with $\zeta$ a primitive third root of unity) by an inclusion chain of diagonal 
	Nichols algebras, namely row 17, row 16 and row 6 for $q
	=-1$:\footnote{Images from Heckenberger's list modified by adding a black dot.}
	\begin{center} 
		\includegraphics[scale=.8]{qDiagrams.png}
	\end{center}
	Hence we \emph{can} obtain Nichols algebras with the right root system and root space 
	dimensions in 
	the category of $\C[x]/(x^3)$-Yetter-Drinfel'd modules. However the Hilbert series' 
	do \emph{not} coincide completely, e.g. in the smallest example we obtain a Nichols algebra 
	with Hilbert series $(2)_t(3)_t(6)_t$ instead of $(2)_t(3)_t(2)_t$ over the group $\S_3$, which 
	suggests to somehow consider a subalgebra of graded index $(3)_{t^2}$. If 
	and how our Nichols algebras above can indeed be turned into Nichols algebras over the 
	respective nonabelian groups is not clear at this point (for the folding in \cite{Len14a} this
	was a central extension, but this cannot be true for $\S_3$). 
\end{example}

\subsection{Restrictions of Nichols algebras}

Let $\B(M^a)$ be the Nichols algebra of a semisimple object
$M^a=M_{\alpha_1}\oplus\cdots\oplus M_{\alpha_n}$ in a braided category $\CC$, and assume for now $\CC=\YD{h}$. Let $(\Ac,R)$ be the associated $n$-dimensional crystallographic arrangement where the chamber $a$ has positive simple roots $\Delta^a=\{\alpha_1,\ldots,\alpha_n\}$.

Fix some $\alpha_i$. Recall that in the construction of
the reflection $r_{\alpha_i}$ in Section \nref{sec_RootSystemsNicholsAlgebras} we have
considered the algebra of coinvariants $\B(\bar{M})$ of the Hopf algebra map $\pi:\B(M^a)\to\B(M_{\alpha_i})$, which is a Hopf
algebra in a different braided category 
$$ \B(\bar{M})
:=\B(M^a)^{coin\;\B(M_i)}\in\YD{\B(M_{\alpha_i})}(\CC),$$
and by Theorem \nref{thm_NicholsAlgebraRootSystem} it is a Nichols algebra $\B(\bar{M})$ in this category and thus has an associated root system. The aim of the next theorem is to show that the root system of $\B(\bar{M})$  is precisely the restricted root system $(\Ac^X,\bar{R})$ with $X=\alpha_i^\perp$. It moreover gives interpretations for the restriction multiplicities and possible non-reducedness of $(\Ac^X,\bar{R})$ in terms of $\B(\bar{M})$. Repeating this argument for a subset $J$ of simple roots, we get a similar statement for an arbitrary intersection of hyperplanes $X=J^\perp$.\\

Let us give in advance a \emph{heuristic argument}, that underlies the proof below: The restriction to $X$ induces a map of sets (additive where appropriate) between the set of roots $R\to \bar{R}\cup \{0\}$ and we denote by $\bar{\alpha}\in \bar{R}^a_+\cup\{0\}$ the image of a root $\alpha\in R^a_+$.
\begin{itemize}
\item The preimage of $0$ are all roots $\alpha$ where $X \subset \alpha^\perp$, i.e.\ in terms of arrangements the localization $\Ac_X$, i.e. in terms of roots the parabolic subsystem generated by $J$. 

In the PBW-basis of $\B(M)=\bigotimes_{\alpha\in R^a_+}\B(M_\alpha)$ the respective factors $\B(M_\beta)$ for all $\beta$ with $\bar{\beta}=0$ hence form precisely the PBW-basis of the subalgebra $\B(M_J)$ with $M_J=\bigoplus_{i\in J} M_{\alpha_i}$.
\item The preimage of any $\bar{\alpha}\neq 0$ is a root $\alpha$ up to addition of roots in $J$. 

Hence \emph{morally}, adding all preimages of $\bar{\alpha}$ should yield an 
irreducible  $\B(M_J)$-Yetter-Drinfel'd submodule 
$$\bar{M}_{\bar{\alpha}}:=\bigoplus_{\beta,\bar{\beta}=\bar{\alpha}}M_\beta$$
of $\bar{M}$. The rank of $\bar{M}$, as 
$h$-Yetter-Drinfel'd module, is by construction equal to the number of 
preimages of $\bar{\alpha}$ and hence the restriction multiplicity.

However, since in general $M_{\alpha_i+\beta}\neq \ad_{M_{\alpha_i}}(M_\beta)$ 
this is \emph{not} true, but we can argue similarly with other 
$\bar{M}_{\bar{\alpha}}$ defined via the right hand side expression.
\item Finally, in the PBW-decomposition of $\B(\bar{M})$ only the roots $\bar{\alpha}$ with $\frac{1}{k}\bar{\alpha}\not\in \bar{R}_+$ appear. The other factors $\B(M_\beta),\bar{\beta}\in \NN \bar{\alpha}$ appear implicitly as higher commutators of $\B(\bar{M}_{\bar{\alpha}})$. This is why we work with the non-reduced system.
\end{itemize}
Finally this yields a PBW-decomposition of $\B(\bar{M})$ according to $(\Ac^X,\bar{R}^a_+)$:
$$\B(\bar{M}_{\bar{\alpha}})=\bigotimes_{\beta,\bar{\beta}\in\NN\bar{\alpha}}\B(M_\beta),
\qquad \B(\bar{M})=\bigotimes_{\bar{\alpha}\in \bar{R}^a_+}\B(\bar{M}_{\bar{\alpha}}).$$

\noindent
We now turn to the exact proof:

\begin{theorem}\label{thm_NicholsAlgebraRestrictionPBW}~
\begin{enumerate}[a)]
\item Assume $M^a=M_{i}\oplus M_{j}$ is of rank $2$ and $\B(M)$ is finite-dimensional. 
Then there is an isomorphism of $\NN$-graded $\YD{\B(M_i)}(\CC)$-objects
$$\bigotimes_{\alpha_i\neq \alpha\in R_+^a} \B(M_\alpha) \cong \B(\bar{M}_{\bar{\alpha}_j}),
\qquad \bar{M}_{\bar{\alpha}_j}:=\ad_{\B(M_{i})}(M_{j}).$$
(Note that the tensor factors are not $\B(M_i)$-Yetter-Drinfel'd modules.)
\item Under the assumptions 
in a) the restricted root systems of $\Ac^{\alpha_i^\perp}$ and of 
$\B(\bar{M})$ are trivially coinciding (type $A_1$), but moreover the 
restriction multiplicity of $\bar{\alpha}_j\in\bar{R}^a_+$ equals the rank of 
the new simple summand $\bar{M}_{\bar{\alpha}_j}$ as semisimple object in $\CC$. 
Non-reducedness, i.e.\ $2\bar{\alpha}_j\in\bar{R}^a_+$, implies 
$[\bar{M}_{\bar{\alpha}_j},\bar{M}_{\bar{\alpha}_j}]\neq 0$.
\item Let $M^a=\bigoplus_{i\in I} M_i$ and assume $\B(M)$ is finite-dimensional with root system $R_+^a$. For $J\subset I$ let $M_J=\bigoplus_{i\in J} M_i$ and let $(\Ac^{X},\bar{R}),X=J^\perp$ be the restricted root system. Then there exists an isomorphism
$$\bigotimes_{\bar{\alpha}\in \bar{R}^a_+} \B(\bar{M}_{\bar{\alpha}}) \cong \B(M^a)^{coin\;\B(M_J)}$$
of $\NN^{I\backslash J}$-graded objects in $\YD{\B(M_J)}(\CC)$
for suitable $\bar{M}_{\bar{\alpha}}\in\YD{\B(M_J)}(\CC)$ where in particular $\bar{M}_{\bar{\alpha}_j}:=\ad_{\B(M_J)}(M_j)$ for $j\not\in J$. Moreover, the restriction multiplicity of $\bar{\alpha}$ determines the rank of the new simple summand $M_{\bar{\alpha}}$ as semisimple object in $\CC$ and non-reducedness, i.e.\ $2\bar{\alpha}\in\bar{R}^a_+$ implies 
$[\bar{M}_{\bar{\alpha}},\bar{M}_{\bar{\alpha}}]\neq 0$.
\item Under the assumptions 
in c) for the Nichols algebra $\B(\bar{M}):=\B(M^a)^{coin\;\B(M_J)}$ in the 
category $\YD{\B(M_J)}(\CC)$, the associated root system (crystallographic 
arrangement) in the sense of the Theorem 
\nref{thm_NicholsAlgebraRootSystem} coincides with the reduced crystallographic arrangement of the restricted root 
system $(\Ac^X,\bar{R})$ with $X=J^\perp$ in the sense of Lemma \nref{lm_ParabolicRestriction}.
\end{enumerate}
\end{theorem}
\begin{proof}~\newline
a) By Theorem \nref{thm_NicholsAlgebraRootSystem} (PBW basis) we have an isomorphism of $\NN^2$-graded $\CC$-objects
$$\bigotimes_{\alpha\in R_+^a} \B(M_\alpha)\cong \B(M),$$
and since the map is multiplication in $\B(M)$ it is an isomorphism in $\YD{\B(M_i)}(\C)$ (of course the tensor factors are no submodules in this matter).\\
By the same theorem we have an isomorphism of $\NN$-graded $\YD{\B(M_i)}(\CC)$-objects 
$$\B(\ad_{\B(M_{\alpha_i})}(M_{\alpha_j}))\cong \B(M)^{coin\;\B(M_i)}.$$
By the Radford projection theorem we have an isomorphism of $\CC$-algebras 
\begin{align*}
\B(M)&\cong\B(M_i)\otimes \B(M)^{coin\;\B(M_i)},\\
x&\mapsto \pi(x^{(1)})\otimes \pi(S(x^{(2)}))x^{(3)},\\
ab&\leftmapsto a\otimes b.
\end{align*} 
and since the $\NN$-grading in the rank $1$ Nichols algebra $\B(\ad_{\B(M_{\alpha_i})}(M_{\alpha_j}))$ is given by the radical filtration and the $\B(M_i)$-action by the adjoint action, this is even an isomorphism of $\NN$-graded $\YD{\B(M_i)}(\CC)$-objects.\\
Altogether we get a surjective map of $\NN$-graded $\YD{\B(M_i)}(\CC)$-objects
$$\bigotimes_{\alpha\in R_+^a} \B(M_\alpha)
\stackrel{\mathrm{mult.}}{\longrightarrow} \B(M) \stackrel{x\mapsto \pi(S(x^{(1)}))x^{(2)}}{\longrightarrow}  \B(M)^{coin\;\B(M_i)}
\cong \B(\ad_{\B(M_i)}(M_j)).$$
We immediately see that $\cdots \otimes \B(M_i)^+$ is in the kernel. Hence this factorizes to the following  surjective map, which is even an isomorphism since the dimensions are equal:
$$\bigotimes_{\alpha_i\neq\alpha\in R_+^a} \B(M_\alpha)
\stackrel{\cong}{\longrightarrow} \B(\ad_{\B(M_i)}(M_j)).$$
b) By definition 
\begin{align*}
\bar{M}&=\ad_{\B(M_i)}(M_j)
=\bigoplus_{k=0}^{\infty} \ad_{M_i}^k(M_j)
=\bigoplus_{k=0}^{-c_{ij}} \ad_{M_i}^k(M_j)
\end{align*}
and by Theorem \nref{thm_NicholsAlgebraRootSystem} all $\ad_{M_i}^k(M_j)$ are simple $\CC$-objects. Hence the rank of $\bar{M}$ as $\CC$-object is $1-c_{ij}$. On the other hand the roots restricting to $\bar{\alpha}_j$ in $\Ac^{\alpha_i^\perp}$ are all elements in the root string $\alpha_j,\alpha_j+\alpha_i,\ldots,\alpha_j-c_{ij}\alpha_i$, so the two numbers coincide as claimed. The claim about non-reducedness comes from considering the dimensions, since $\bar{M}$ generates $\B(\bar{M})$.\\
\\
c) Since both parabolic restriction and taking coinvariants can be performed root-by-root it 
suffices to treat the case $J=\{\alpha_1\}$.
As in a) we have by the PBW basis theorem an isomorphism of $\NN^{I\backslash J}$-graded objects in $\YD{\B(M_J)}(\CC)$
$$\bigotimes_{\alpha\in{R}^a_+} \B(M_{{\alpha}}) \cong \B(M^a)$$
and by the Radford projection theorem an isomorphism of 
$\NN^{I\backslash J}$-graded objects in $\YD{\B(M_J)}(\CC)$
$$\bigotimes_{\alpha\in {R}^a_+\backslash R^a_{J,+}} \B(M_{{\alpha}}) \cong \B(M^a)^{coin\;\B(M_J)}$$
where $R^a_{J,+}:=R^a_{+}\cap J\NN$ is the parabolic subsystem (localization). Now restriction to $(\Ac^{J^\perp},\bar{R})$ defines a surjective map (even compatible with addition where appropriate)
$${R}^a_+\backslash R^a_{J,+} \to \bar{R}^a_+$$
where the preimages of $\bar{\alpha}_j$ is the root string $(\alpha_j+\NN\alpha_i)\cap R^a_+$. Similarly all other preimages of $\bar{\alpha}$ are root strings after reflection, i.e. in some other object $a'$ where $\alpha$ is in a parabolic subsystem $\langle\alpha_i,\alpha_j\rangle$. Reordering (which is clearly isomorphism) and using for each $\bar{\alpha}$ the rank $2$ result of a) yields the desired isomorphism. The claims on multiplicities and nonreducedness follow from c).\\
d) Since we obtained in c) a decomposition 
$$\bigotimes_{\bar{\alpha}\in \bar{R}^a_+} \B(M_{\bar{\alpha}})\cong \B(\bar{M})$$
of the Nichols algebra in $\NN^{I\backslash J}$-graded $\B(M_J)$-Yetter-Drinfel'd modules, and the degrees of this decomposition is by construction given by $\Ac^X$, we can apply the uniqueness result in \cite{HS10} Theorem 4.5 (2).\footnote{We thank Jing Wang for pointing out this uniqueness result for the root system.}
\end{proof}

\subsection{First examples of restrictions}\label{sec_examples}

We first discuss very explicitly two small examples of parabolic restriction of 
diagonal Nichols algebras, which yield familiar restricted root 
systems. 

For the convenience of the reader we first compare the notation for the 
diagonal case in Lusztig's book on quantum groups \cite{Lusz94} and 
Heckenberger's paper \cite{Heck07} on reflections in diagonal Nichols algebras:
\begin{center}
\begin{tabular}{p{7.5cm}p{.8cm}p{6cm}}
 {\bf Diagonal case} && {\bf General case} \\
 \hline
 Braided vector space\newline
 spanned by $\theta_i,\; E_i$ or $x_i$ for $i\in I$\newline
 $V=V^+(\chi)$ depending on a bicharacter $\chi$.
 && Yetter-Drinfel'd module $M$\\
 \hline
 Diagonal braiding \newline
 $x_i\otimes x_j\to q_{ij}\; x_j\otimes x_i$\newline
 braiding matrix $q_{ij}=v^{i\cdot j/2}=q^{d_ic_{ij}}$\newline
 or more generally $q_{ij}=\chi(\alpha_i,\alpha_j)$
 && Yetter-Drinfel'd module braiding\newline
 $a\otimes b\mapsto a^{(-1)}.b\otimes a^{(0)}$\\
 \hline
 1-dim. subspaces
 $\theta_i\C,\;E_i\C,\;x_i\C$
 \newline Duals $F_i\C$ in $V^*=V^-(\chi)$ 
 && Simple summands $M_i$ of $M$
 \newline Duals $M_i^*$ in $M^*$\\
 \hline
 Nichols algebra $\mathbf{f},\;\B(V)$ or $\mathcal{U}^+(\chi)/\mathcal{I}^+$
 && Nichols algebra $\B(M)$\\
 \hline
 Cartan matrix\newline
 $-c_{ij}=\max\{m\;|\;\ad_{x_i}^m(x_j)\neq0\}$\newline
 $=\min\{m\mid (m+1)_{q_{ii}}=0 \text{ or }
  q_{ii}^m q_{ij}q_{ji}=1\}$
 && Cartan matrix\newline
 $-c_{ij}:=\max\{m\;|\;\ad_{M_i}^m(M_j)\neq0\}$\\
 \hline
 Subalgebras $\mathbf{f}[i]$ in \cite{Lusz94} Sec. 38.1.1 \newline
 generated by all $\ad^m_{\theta_i}(\theta_j)$ for fixed $i\in I$ \newline 
 Subalgebra $\mathcal{U}^+_{+i}(\chi)$ modulo $\mathcal{I}^+$ for $i\in I$\newline
 in \cite{Heck07} Def. 4.29, Prop. 5.10
 && Coinvariants over $\B(M_i)$ for $i\in I$\newline
 $K_i:=\B(M)^{coin\;\B(M_i)}$\\
 \hline
 Set of positive roots $R^+,\;\mathcal{R}_+^a$\newline
 Reflection operator $r_i,\;T_i$ for $i\in I$ && 
 Set of positive roots $\mathcal{R}_+^a$\newline
 Reflection operator $T_i$ for $i\in I$\\
 \hline
\end{tabular}
\end{center}~\\

\begin{example}\label{exm_RestrictionA2}
   Let $M_{1}=x_1\C$ and $M_2=x_2\C$ be Yetter-Drinfel'd modules over
  $\Gamma:=\Z_2\times\Z_2=\langle g_1,g_2\rangle$ with
  $$\chi_1(g_1)=-1,\qquad \chi_2(g_1)=-1,$$
  $$\chi_1(g_2)=1,\qquad \chi_2(g_2)=-1.$$
  Then $M:=M_1\oplus M_2$ has diagonal braiding 
$$q_{ij}=\begin{pmatrix} -1 & -1 \\ 1 & -1\end{pmatrix},$$ 
and the Nichols algebra 
$\B(M)$ is of Lie type $A_2$,
	$$R_+^a=\{(1,0),(0,1),(1,1)\}.$$
	This means that multiplication
  in the Nichols algebra induces the following isomorphism of graded vector 
spaces (PBW-basis)
  $$\B(x_1)\otimes \B(x_{12})\otimes \B(x_2)\cong \B(M)
  \qquad x_{12}:=x_1x_2+x_2x_1\in \B(M).$$
  The Hilbert series is hence $(2)_t^2(2)_{t^2}$ and the dimension is $2^3$.\\
  
  Consider the coinvariants $K_1=\B(M)^{coin\;\B(M_1)}$ with respect to the 
projection 
	$\B(M_1\oplus M_2)\to\B(M_1)$, it is generated by 
$x_2,x_{12}$ (this is an easy direct calculation). By Theorem 
	\nref{thm_NicholsAlgebraRootSystem},
  $K_1$ is the Nichols algebra of the $\B(M_1)$-Yetter-Drinfel'd 
module\footnote{inside the base category $\YD{\Gamma}$ i.e. $K_1$ is actually a 
$\B(M_1)\rtimes k[\Gamma]$-Yetter-Drinfel'd module.} 
  $$\bar{M}=\ad_{\B(M)}(N)=\ad_{1}(x_2)+\ad_{x_1}(x_2)=x_2\C\oplus 
x_{12}\C.$$
  $\bar{M}$ is simple as a $\B(M_1)$-Yetter-Drinfel'd module 
(hence the root system of $\B(\bar{M})$ is of type $A_1$) and 
$\bar{M}$ is semisimple of rank $2$ as an object in the base category 
$\YD{\Gamma}$ (since $\Gamma$ is abelian, this is equal to the $\C$-dimension).

For the convenience of the reader we write down the 
explicit \emph{nondiagonal braiding} on $\bar{M}$ as it follows from 
$\B(M_1)$-action and -coaction; observe that the off-diagonal term comes from a 
$2x_1$ coacting from $x_{12}$ and then acting on $x_2$:  

$$\begin{pmatrix}
 x_2\otimes x_2 & x_2 \otimes x_{12} \\
 x_{12} \otimes x_2 & x_{12} \otimes x_{12}
\end{pmatrix}
\;\longmapsto\;
\begin{pmatrix}
 \hspace{2cm}-x_2\otimes x_2 & -x_{12} \otimes x_{2} \\
 -2x_{12} \otimes x_2 +x_2 \otimes x_{12} & -x_{12} \otimes x_{12}
\end{pmatrix}$$

This is in agreement with our result. Namely, the restriction of the root 
system $(\Ac,R)$ of type $A_2$ to the hyperplane $H_{\alpha_1}$ is the root 
system $(\Ac^{H_{
	\alpha_1}},\bar{R})$ of type $A_1$. More precisely, $\alpha_1=(1,0)$ 
restricts to $0\in \bar{R}$ and is removed, while both $\alpha_2,\alpha_{12}$ 
restrict to the same simple root  in $\bar{R}$ (multiplicity $2$).

\begin{center} 
		\includegraphics[scale=.25]{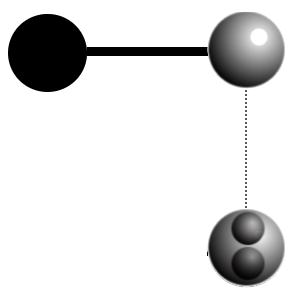}
\end{center}
    
As indicated in the application part of the introduction, this $\bar{M}$ is the 
fundamental $\sl_2$-representation used in \cite[Sec.\ 4]{Rosso98} as an 
induction step from $u_q(\sl_2)$ to $u_q(\sl_3)$. In the same way our result 
explains the different $A_{n-1}$ representations appearing in Rosso's induction 
step from $A_{n-1}$ to $A_n,B_n,C_n,D_n, etc.$ 
\end{example}

We next consider an example which is not of Lie type:

\begin{example}
	Consider the diagonal Nichols algebra $\B(M)$ of rank $3$, thoroughly 
treated in 
	Example \nref{exm_CycleNicholsAlgebra}, which has braiding matrix 
	$$q_{\alpha_i\alpha_i}=-1,\qquad 
q_{\alpha_i\alpha_j}q_{\alpha_j\alpha_i}=\zeta,$$  
	and the following root system and vector space PBW-basis:
	\begin{align*}
	R^a_+&=\{(1,0,0),(0,1,0),(0,0,1),(1,1,0),(1,0,1),(0,1,1),(1,1,1)\},\\
	\B(M)&\cong \B(x_{\alpha_1})\otimes \B(x_{\alpha_2})\otimes 
\B(x_{\alpha_3})
	\otimes \B(x_{\alpha_{12}})\otimes \B(x_{\alpha_{13}})\otimes 
\B(x_{\alpha_{23}})
	\otimes \B(x_{\alpha_{123}})\\
	&\cong \C[x_{\alpha_{1}}]/(x_{\alpha_{1}}^2)\otimes
	\C[x_{\alpha_{2}}]/(x_{\alpha_{2}}^2)\otimes 
	\C[x_{\alpha_{3}}]/(x_{\alpha_{3}}^2)\\
	&\otimes\C[x_{\alpha_{12}}]/(x_{\alpha_{12}}^3)\otimes
	\C[x_{\alpha_{13}}]/(x_{\alpha_{13}}^3)\otimes
	\C[x_{\alpha_{23}}]/(x_{\alpha_{23}}^3)\\
	&\otimes\C[x_{\alpha_{123}}]/(x_{\alpha_{123}}^2).
		\end{align*}
	Consider the projection to the Nichols algebra 
	$$\B(M_{\alpha_2})\cong \C[x_{\alpha_2}]/(x_{\alpha_2}^2).$$ 
	The action and coaction of $x_{\alpha_2}$ span the following simple 
$2$-dimensional 
	Yetter-Drinfel'd submodules over $\B(M_{\alpha_2})\rtimes\C[\Gamma]$: 
		\begin{align*}
	M_{\bar{\alpha_1}}&:=M_{\alpha_1}\oplus M_{\alpha_{12}},\\
	M_{\bar{\alpha_2}}&:=M_{\alpha_3}\oplus M_{\alpha_{23}},\\
	M_{\bar{\alpha_{12}}}&:=M_{\alpha_{13}}\oplus M_{\alpha_{123}}.
	\end{align*} 
	We obtain the new generating Yetter Drinfel'd module according to 
	Theorem \nref{thm_NicholsAlgebraRootSystem}:
	$$\B(M)^{coin\;\B(M_{\alpha_2})} = \B(\bar{M}),\qquad
	\bar{M}=\ad_{\B(M_{\alpha_2})}(M_{\alpha_1}\oplus M_{\alpha_3})=
	M_{\bar{\alpha_1}}\oplus M_{\bar{\alpha_2}}.$$
	Altogether, this agrees with the crystallographic arrangement 
$(\Ac^{H_{\alpha_2}},\bar{R})$ which is of type $A_2$ 
	and all multiplicities $2$. We get a new PBW-basis and Hilbert series
	\begin{align*}
	\B(\bar{M})&\cong \B(M_{\bar{\alpha}_1})\otimes \B(M_{\bar{\alpha}_2})
	\otimes\B(M_{\bar{\alpha}_{12}}),\\
	\mathcal{H}_{\B(\bar{M})}(t) &=(2)_{t}^2(3)_{t}^2(3)_{t^2}(2)_{t^2}.
	\end{align*}
\end{example}

\subsection{Nichols algebras with new Weyl groupoids}\label{sec_newWeylGroupoids}
 
Recall from Corollary \nref{cor_NicholsAlgebraExistingRootsystems} that only certain root systems appear as root systems of Nichols algebras in $\YD{\C[G]}$. Parabolic restriction produces Nichols algebras in other braided categories 
$$\B(\bar{M}):=\B(M)^{coin\;\B(M_J)}\;\in\; \YD{\B(M_J)}(\CC)$$
and by Theorem \nref{thm_NicholsAlgebraRestrictionPBW} the root system and Weyl groupoid corresponds to the restricted root system $(\Ac^X,\bar{R})$, $X=J^\perp$. We hence find an answer to Question \nref{q_Cuntz}:

\begin{theorem}\label{thm_newWeylGroupoid}~ 
\begin{enumerate}
\item There exist Nichols algebras whose corresponding crystallographic arrangement are the sporadic arrangements of rank three labeled $7,13,14,15,20,23$, although such a Nichols algebra does not exist over any finite group. These have $13,16,17,17,19,19$ positive roots respectively, and are restrictions of the reflection arrangements of type $E_7,E_8,E_8,E_8,E_8,E_8$ respectively.
\item Since every crystallographic arrangement of rank greater than three is a (parabolic) restriction of a Weyl arrangement, every crystallographic arrangement of rank greater than three is symmetry structure of some Nichols algebra.
\end{enumerate}
\end{theorem}

\begin{question}
  Do all root systems appear as root systems of some Nichols algebra in some
  braided category $\CC$?
\end{question}

We now give first an example with a new sporadic Weyl groupoid of rank three labeled $7$:

\begin{example}\label{exm_newWeylGroupoid}
Let $\B(M)=u_q(E_7)^+$, which is a diagonal Nichols algebra of dimension 
$\ord(q^2)^{63}$. We consider the parabolic restriction indicated in the 
diagram:
\begin{center}
\includegraphics[scale=.6]{RestrictionE7To7.png}
\end{center}

Then the restriction of the Weyl group arrangement $E_7$ with $63$ roots has a 
Weyl groupoid of sporadic type $7$ and $13$ roots (one of which is nonreduced),  
namely\\

\begin{center}
\begin{tabular}{l|lllllllllllll}
$R_+^a$ 				&	$\bar{\alpha}_3$&$\bar{\alpha}_4$	&$\bar{\alpha}_5$ &	$(1,1,0)$ & $(0,1,1)$	&
$(1,1,1)$ & $+2(1,1,1)$	\\
multiplicity 		& $2$							& $2$							& $3$							&	$4$				& $6$ 			& 
$12$ & +3\\
\hline
$R_+^a$ 				& $(1,2,1)$	& $(2,2,1)$	& $(1,2,2)$	& $(2,3,2)$	& $(2,3,3)$	& $(2,4,3)$&$(3,4,3)$\\
multiplicity 		&	$6$				& $3$				& $6$				& $6$				& $2$				& $1$ 			& $2$\\
\end{tabular}
\end{center}~\\

This yields a Nichols algebra $\B(\bar{M})$ of dimension $\ord(q^2)^{58}$ over 
an object $\bar{M}$ of rank $3$ and dimension $2+2+3$ in the braided category 
$\YD{\B(M_J)}$ with $J=\{\alpha_1,\alpha_2,\alpha_6,\alpha_7\}$, where $\B(M_J)$ 
is an ordinary Borel part $u_q(A_1\times A_1\times A_2)^+$.\\

This root system is clearly not one of Lie type. This means that reflections in 
the Weyl groupoid connect the Weyl chamber described above with other chambers 
corresponding possibly to different Dynkin diagrams and non-isomorphic Nichols 
algebras (compare Example \nref{exm_CycleNicholsAlgebra}).\\

For the restriction these different choices arise from different parabolics 
$J'$ of $E_7$ in the Weyl group orbit of the choice $J$ above. To be more 
precise in our example: The sporadic Weyl groupoid of rank three labeled $7$ 
has by \cite{p-CH10} Table 1 altogether $12$ different\footnote{To be precise, 
this table only shows $12$ different root sets. Counting different 
multiplicities of the roots, the number of objects could a-priori multiply 
up to $96$ (simply connected Weyl groupoid). But in this example,  
there are $12$ parabolics $J'\cong A_1\times A_1\times A_2$ 
restricting to no.\ $7$.} chambers (objects) 
with the following Dynkin diagrams
% (see Fig.\ \ref{ocd37} for the object change diagram):
(see the Figure below for the object change diagram):
$A_3,B_3,C_3,\mathcal{A}_3^1(2)$
% \mbox{\it, and the sporadic Dynkin diagrams }
and the sporadic Dynkin diagrams
$\Gamma_3^1,\Gamma_3^7,\Gamma_3^{11},\Gamma_3^{15}$ in \cite{p-CH10}.\\
% \mbox{\it~ in \cite{p-CH10}}.

How do these different objects in the restriction arise, even 
though $E_7$ itself has only one type of chamber? This is because the parabolic 
$J$ is not fixed by the Weyl group of $E_7$. For example, the reflections in 
$E_7$ on the leftmost nodes $s_{\alpha_3}s_{\alpha_1}s_{\alpha_3}$ map $J$ 
above to the following $J'$. In the parabolic restriction this changes the 
object (chamber) above with Dynkin diagram $A_3$ to a different object 
(chamber) with Dynkin diagram $C_3$. In particular the now simple root with 
multiplicity $4$ is the root $(1,1,0)$ above and the double edge arises from the 
root string $(0,0,1),(1,1,1),(2,2,1)$.
\begin{center}
\includegraphics[scale=.6]{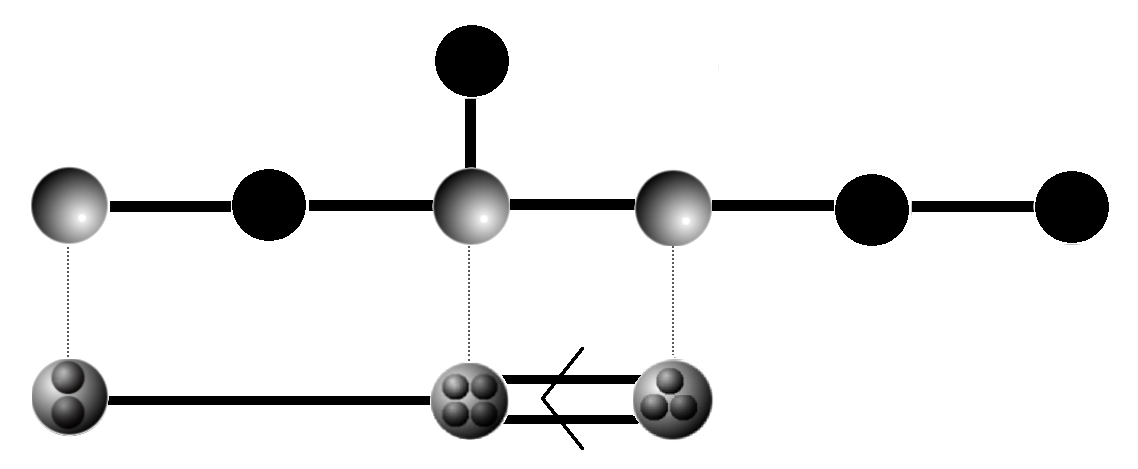}
\end{center}

% Wenn man Typ E7 aus der Cartan-Matrix
% [ 2  0 -1  0  0  0  0]
% [ 0  2  0  0 -1  0  0]
% [-1  0  2  0  0 -1  0]
% [ 0  0  0  2  0  0 -1]
% [ 0 -1  0  0  2  0 -1]
% [ 0  0 -1  0  0  2 -1]
% [ 0  0  0 -1 -1 -1  2]
% berechnet, dann ist das Arrangement Nummer 7 mit 13 Hyperebenen im Rang drei 
% Restriktion auf folgende Koordinatentupel:
% [1,2,7],[1,3,7],[1,5,6],[1,5,7],[1,6,7],[2,6,7],[3,4,5],[3,4,7],[3,5,6],[3,6,7],
% [4,5,6],[5,6,7].
% Entspricht den 12 objekten in \cite{p-CH10}
\end{example}

We finally present the full \emph{object change diagram} of the sporadic 
Weyl groupoid of rank three labeled $7$. It shows Cartan matrix and Dynkin 
diagram in each object (chamber) and the edges are labeled by reflections on 
the $i$-th simple root (simple roots in the standard ordering, differently from 
the ordering above!). Reflections not drawn map to the same object. The thick 
arrow shows the reflection of the $J$-restriction to the $J'$-restriction 
discussed above.

\begin{center}
\includegraphics[scale=.3]{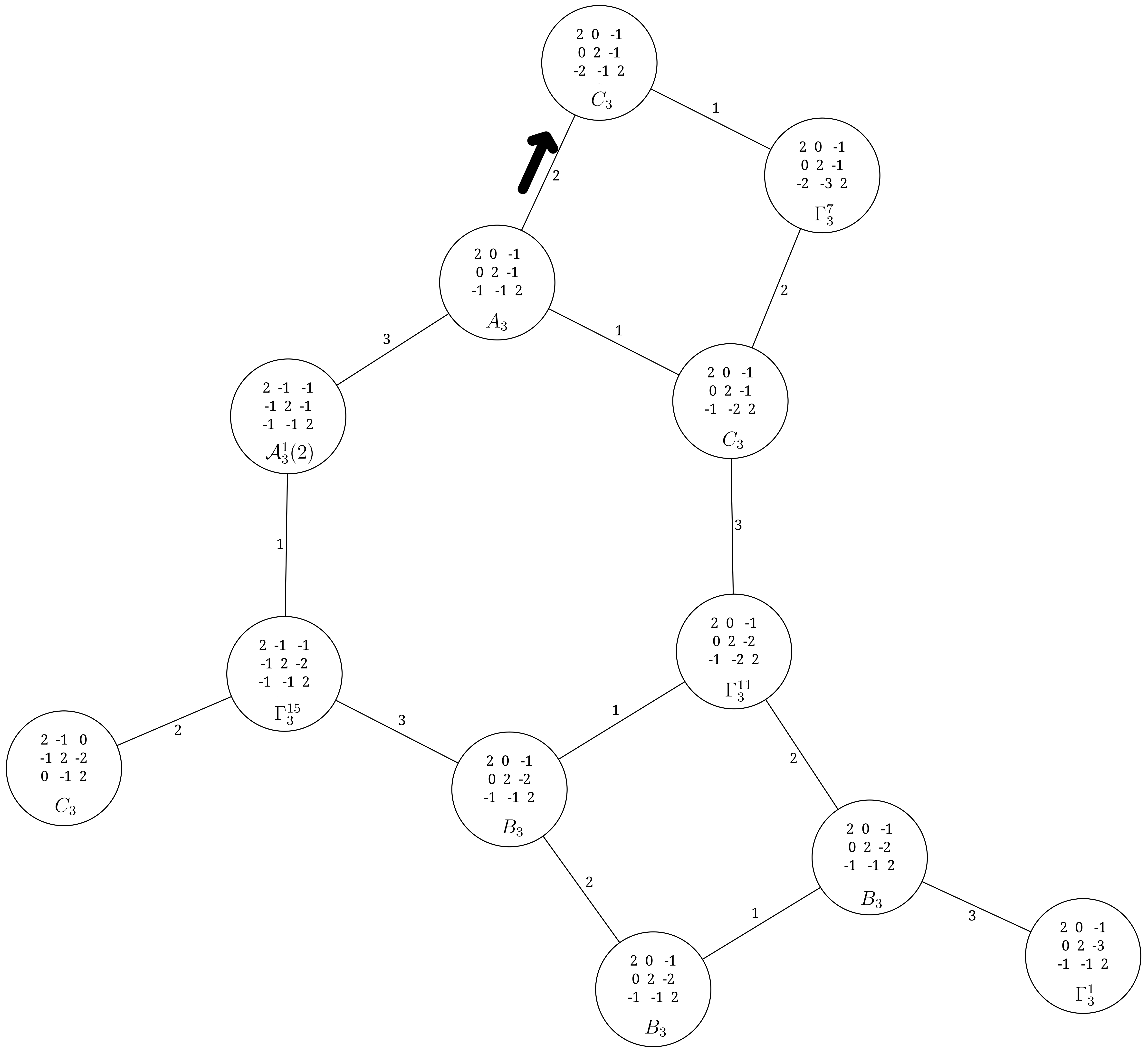}
\end{center}
% \begin{figure}
% \caption{The object change diagram of the sporadic Weyl groupoid of rank three 
% labeled $7$.}
% \label{ocd37}
% \end{figure}

\section{A simplicial complex of Nichols algebras}\label{simpcomp_section}

\subsection{The simplicial complex}

As in \cite{CMW15}, to a crystallographic arrangement $(\Ac,R)$ of rank $n$ we associate the \emph{simplicial complex}
$S$ on $V$ whose $(n-1)$-cells are the chambers of $\Ac$ intersected with the $n-1$-sphere.
This complex $S\subseteq V$ is topologically isomorphic to a sphere.
Sometimes it will be convenient to add a single $(-1)$-cell corresponding to the origin $0\in V$.
We need the following facts (see \cite{CMW15} for proofs and more details):

\begin{lemma}\label{lm_SimplicialComplex}~
\begin{enumerate}[a)]
\item Let $K$ be a chamber and $\Delta^K$ be the associated set of simple roots. Then the 
adjacent $(k-1)$ cells $x\subseteq K$ are in bijection with the subsets $J_{x,K}\subseteq \Delta^K$ of order $n-k$.
Moreover $x\subseteq X_x:=J_{x,K}^\perp$ is of dimension $k$.
\item Let $K_1\ne K_2$ be chambers and let $w$ be the Weyl groupoid element with $w.K_1=K_2$.
Then the intersection $K_1\cap K_2$ is a common adjacent $(k-1)$-cell 
$x\subseteq K_1,K_2$, $w$ is contained in the Weyl groupoid of the localization
$\Ac_{X_x}$, and $w.J_{x,K_1}=J_{x,K_2}$.
\end{enumerate}
\end{lemma}

\subsection{The decoration by Nichols algebras in braided categories}

Suppose we have a Nichols algebra $\B(M)$, for now in the braided category
$\CC=\YD{h}$. We consider the crystallographic arrangement
$(\Ac,R)$ with roots $\alpha\in R$ labeled by objects $M_\alpha\in \CC$ from Theorem
\nref{thm_NicholsAlgebraRootSystem}. More precisely, the initial Nichols algebra is $\B(M^a)$ associated to some chamber $a$ of $\Ac$ and one set of positive simple roots $\Delta_a^+$ such that $M^a=\bigoplus_{\alpha\in \Delta_a^+}M_\alpha$, while the other chambers correspond to Weyl equivalent Nichols algebras $\B(M^{a'})$. Thus it would be better to think of $(\Ac,R)$ being associated to the entire Weyl equivalence class and a specific representing Nichols algebra $\B(M^a)\in \CC$ being associated to each $(n-1)$-cell $a$. We now wish to associate Nichols algebra data to all $(k-1)$-cells $x$.

\begin{definition}
	As in Lemma \nref{lm_SimplicialComplex}, let $x$ be a $(k-1)$-cell, let $X_x\in L(\Ac)$ 
	the associated subspace of $V$, let $a$ be any adjacent chamber with simple roots 
	$\Delta^a$ and $J_{x,a}\subseteq \Delta^a$.\\
	Then we associate to the pair $(x,a)$ the braided category of the localized Nichols 
	algebra 
	$$\CC_x:=\YD{\B(M_{x,a})},\qquad M_x:=\bigoplus_{\alpha\in J_{x,a}}M_\alpha$$
	and the restricted Nichols algebra in this category as in Theorem \nref{thm_NicholsAlgebraRestrictionPBW}
	$$\B(M)^x:=\B(M^a)^{coin\,\B(M_{x,a})}\cong \B(M^{x,a}),\qquad  M^{x,a}
	:=\ad_{\B(M_{x,a})}\left(\bigoplus_{\alpha\in \Delta^a\backslash J_{x,a}}M_\alpha\right).$$
\end{definition}

\begin{remark}
The previous definition is \emph{independent} of the chosen adjacient chamber $a$ in the 
following sense: Let $a_1,a_2$ be chambers adjacent to $x$, then by Lemma \nref{lm_SimplicialComplex} b) we have an element $w$ in the Weyl groupoid of $\Ac$, which is contained in the localization to $X_x$, such that $w.a_1=a_2$ and $w.J_{x,a_1}=J_{x,a_2}$. Thus $R_w(M_{x,a_1})=M_{x,a_2}$ and the Nichols algebras $\B(M_{x,a_1}),\B(M_{x,a_2})$ are Weyl equivalent. By \cite{BLS14} Thm. 4.4 this implies there is a category isomorphism 
$$r_w:\;\YD{\B(M_{x,a_1})}\cong \YD{\B(M_{x,a_2})},$$
$$r_w:\;\B(M^{x,a_1})\mapsto \B(M^{x,a_2}).$$
In particular, the vector space, the root system and Dynkin diagram, and the vector space dimension of the nodes $\bigoplus_i M_{\bar{\alpha}_i}=M^{x,a}$ (multiplicities) are independent of $a$ (in contrast to $M_{x,a}$).
\end{remark}

\begin{corollary}
In this sense, the simplicial complex is decorated by sets of equivalent Nichols algebras in an equivalence class of braided categories and by Dynkin diagrams and root multiplicities.
\end{corollary}

We first give the two extremal examples:

\begin{example}
  The chambers $x=a\in A$ are the $(n-1)$-cells. Each corresponds to the empty	
  subsets $J_{a,a}=\{\}$ and is contained only
  in the chamber $a$ itself. It is hence decorated with the unique category 
  $$\CC_a=\YD{\B(M_{x,a})}=\YD{1_\CC}=1_\CC$$
  and the following Nichols algebra in this category:
  $$\B(M)^a=\B(M^a)^{coin\,1_{\B(M^a)}}=\B(M^a)\in \CC.$$
  Hence as expected we associate to each $(n-1)$-cell $a$ the Nichols algebra $\B(M^a)$ in 
	the base category $\CC$.
\end{example}
\begin{example}
  The center point $0$ is a $(-1)$-cell adjacent to any chamber $a$. It corresponds to 
	the full set of simple roots depending on the chamber $J_{0,a}=\Delta^a$
  and is hence decorated with the equivalence class of categories
	$$\CC_0=\YD{\B(M_{0,a})}=\YD{\B(M^a)}$$
  and the following Nichols algebras in each category:
  $$\B(M)^0=\B(M^a)^{coin\,\B(M^a)}=1_{\YD{\B(M^a)}}.$$
  Hence we associate to the $(-1)$-cell in the center the trivial Nichols algebras, i.e.\ on 
	unit objects in all braided categories $\YD{\B(M^a)}$ over some Weyl
  equivalence representative of the full Nichols algebra $\B(M^a)$.
\end{example}

We also note that \emph{simple reflection} can be visualized nicely from this picture: Let $a_1,a_2$ be adjacent chambers, meaning $a_1\cap a_2$ is an $(n-2)$-cell, i.e.\ in a unique hyperplane $H_\alpha,\alpha\in\Delta^{a_1}$. Then the reflection $r_\alpha:\B(M^{a_1})\to\B(M^{a_2})$ constitutes of the following steps:
\begin{itemize}
\item Restriction to the hyperplane $H_\alpha$ yields the restricted Nichols algebra $K:=\B(M^{a_1})^{coin\,\B(M_\alpha)}\in \YD{\B(M_\alpha)}=C_x$. 
\item The orientation chance $H_\alpha=H_{-\alpha}$ i.e.\ $J_{x,a_1}=\{\alpha\}\mapsto \{-\alpha\}=J_{x,a_2}$ is simply a dualization $\B(M_\alpha)\mapsto \B(M_\alpha^*)$, yielding equivalent categories $C_x=\YD{\B(M_\alpha)}\cong \YD{\B(M_\alpha)^*}$ and especially as image of $K$ a Nichols algebra $L\in \YD{\B(M_\alpha^*)}$.
\item Now $K$ is the restriction of the reflected Nichols algebra $\B(M^{a_2})$ to the hyperplane $x=H_{-\alpha}$ and we may conversely obtain the reflected Nichols algebra from $K$ by a Radford biproduct.  
\end{itemize}
 
\subsection{Example: A Nichols algebra of rank 3}\label{sec_ExampleSimplicialComplex}

Consider the diagonal Nichols algebra thoroughly treated in Example 
\nref{exm_CycleNicholsAlgebra}. It was of rank $3$, and had two types of 
chambers $a,a'$ with the following sets of $7$ positive roots

$$R^a_+=\{(1,0,0),(0,1,0),(0,0,1),(1,1,0),(1,0,1),(0,1,1),(1,1,1)\}$$
$$R^{a'}_+=\{(1,0,0),(0,1,0),(0,0,1),(1,1,0),(0,1,1),(1,1,1),(1,2,1)\}$$
$$\mbox{where }\alpha_1'=\alpha_{12},\alpha_2'=-\alpha_2,\alpha_3'=\alpha_{23}$$

The following picture now shows the arrangement in perspective: It is basically a cuboctahedron (an archimedian solid) with $8$ equilateral triangles and $6$ squares, which has each square subdivided into $4$ right triangles (by the three orthogonal planes $H_{\alpha_{ij}}$). There are hence $7$ hyperplanes and $8+24=32$ chambers. 

\begin{center}
\includegraphics[scale=.5, angle=0]{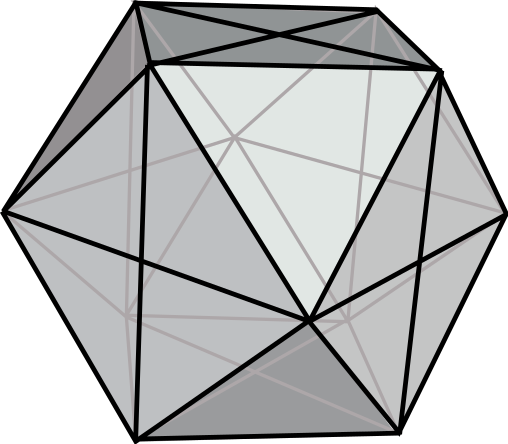}
\includegraphics[scale=.25]{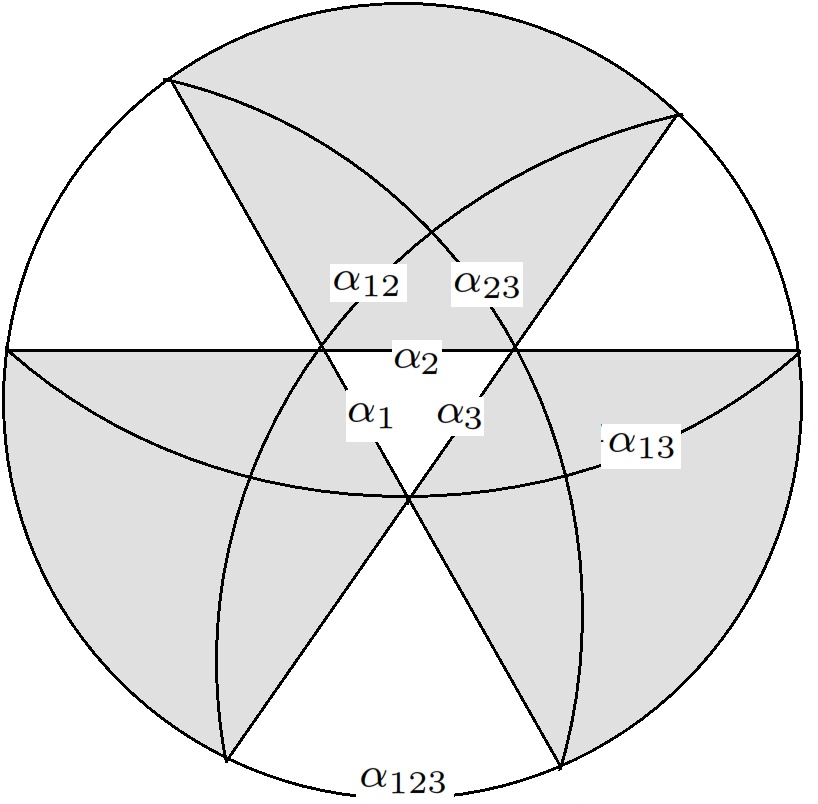}
\end{center}
\noindent
So we have a spherical simplicial complex of dimension $3-1$. We now calculate the categories and Nichols algebras associated to each $(k-1)$-cell:
\begin{itemize}
\item $2$-Cells: There are two types of chambers: 
	\begin{itemize}
	\item There are $8$ chambers associated to equilateral triangles as the example $x=a$ 
	bounded by $H_{\alpha_1},H_{\alpha_2},H_{\alpha_3}$. It is associated to the Nichols 
	algebra $\B(M^a)$ in the base category $\CC_x=\CC$ with Dynkin diagram a cycle and 
	Hilbert series 
	$(2)_{t}^3(3)_{t^2}^3(2)_{t^3}$
	\begin{center}
  \includegraphics[scale=.6]{DynkinZ3.png}
  \end{center}
	\item There are $24$ chambers associated to right triangles as in the example $x=a'$ 
	bounded by $H_{\alpha_{12}},H_{-\alpha_2},H_{\alpha_{23}}$ (note $H_{-\alpha_2}=H_{\alpha
	_2}$). It is associated to the Nichols algebra $\B(M^{a'})$ in the base category $\CC_x=
	\CC$ with Dynkin diagram $A_2$ (but one more root) and Hilbert series 
	$(2)_{t}(3)_{t}^2(2)_{t^2}^2(2)_{t^3}(3)_{t^4}$.
	\begin{center}
  \includegraphics[scale=.6]{DynkinA3.png}
  \end{center}
	\end{itemize}
\item $1$-Cells: There are two types of edges:
	\begin{itemize}
	\item There are $24$ edges between an equilateral and right triangle (the edges of the 
		cuboctahedron) as in the examples $a,a'$ the common hyperplane $x=H_{\alpha_2}$. It is 
		associated to the following restricted Nichols algebra $\B(M^{x,a})$ in the category 
		$$\CC_x=\YD{\B(M_{\alpha_2})}, \mbox{ where }\B(M_{\alpha_2})
		\cong\C[t_{\alpha_2}]/(t_{\alpha_2}^2)$$ 
		Note that we calculate the restriction from the chamber $a$, but the calculation from 
		$a'$ would return the same result except $\alpha_2\to-\alpha_2$. The generating Yetter-
		Drinfel'd modules is $M^{x,a}=M^{x,a}_1\oplus M^{x,a}_2$, the Nichols algebra is of 
		Lie type $A_3$. More precisely, the simple summands are
		$$M^{x,a}_1=M_{\alpha_1}\oplus M_{\alpha_{12}}\qquad
		M^{x,a}_2=M_{\alpha_3}\oplus M_{\alpha_{23}},$$
		$$[M^{x,a}_1,M^{x,a}_2]=M_{\alpha_{13}}\oplus M_{\alpha_{123}}.$$
		In particular, the Hilbert series of $\B(M^{x,a})$ is  
		$(2)_{t}^2(3)_{t}^2(3)_{t^2}(2)_{t^2}$.
		\begin{center}
		\includegraphics[scale=.6]{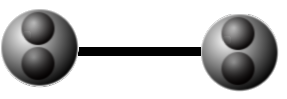}
		\end{center}
		\item There are $24$ edges between two right triangles (half subdividing diagonals in 
		the cuboctahedron squares) as in the examples $a'$ the hyperplane $x=H_{\alpha_{23}}$. 
		It is 
		associated to the following restricted Nichols algebra $\B(M^{x,a})$ in the category 
		$$\CC_x=\YD{\B(M_{\alpha_{23}})}, \mbox{ where }\B(M_{\alpha_{23}})
		\cong\C[t_{\alpha_{23}}]/(t_{\alpha_{23}}^3)$$
		The generating 
		Yetter-Drinfel'd modules is $M^{x,a}=M^{x,a}_1\oplus M^{x,a}_2$, the Nichols algebra 	
		is of type $B_2$. More precisely, the simple summands are
		$$M^{x,a}_1=M_{-\alpha_{2}}\oplus M_{\alpha_{1}}\qquad
		M^{x,a}_2=M_{\alpha_{23}}$$
		$$[M^{x,a}_1,M^{x,a}_2]=M_{\alpha_{3}}\oplus M_{\alpha_{123}}\qquad
		[M^{x,a}_1,[M^{x,a}_1,M^{x,a}_2]]=M_{\alpha_{13}}$$
		In particular, the Hilbert series of $\B(M^{x,a})$ is  
		$(2)_{t}^2(3)_{t}(2)_{t^2}^2(3)_{t^3}$.
		\begin{center}
		\includegraphics[scale=.6]{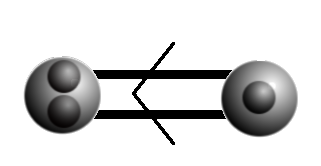}
		\end{center}
	\end{itemize}
	\item $0$-Cells: There are two types of $0$-cells: 
	\begin{itemize}
		\item There are $12$ vertices with altogether three intersecting hyperplanes (the 
		cuboctahedron vertices) as in the examples $a$ the bounding hyperplanes 
		$x=H_{\alpha_1}\cap H_{\alpha_3}$. It is 
		associated to the following restricted Nichols algebra $\B(M^{x,a})$ in the category 
		$$\CC_x=\YD{\B(M_{\alpha_{1}}\oplus M_{\alpha_3})}$$ 
		where $\B(M_{\alpha_{1}}\oplus M_{\alpha_3})$ is a Nichols algebra of dimension 
		$2^2 3=12$ with Dynkin diagram 
		$A_2$ (the quantum Borel part of $\mathfrak{sl(2|1)}$ at the 
third root of unity $q=\zeta$): The generating 
		Yetter-Drinfel'd module $M^{x,a}$ is simple, the Nichols algebra of type $A_1$: 
		$$M^{x,a}=M_{\alpha_{2}}\oplus M_{\alpha_{12}}
		\oplus M_{\alpha_{23}}\oplus M_{\alpha_{123}}$$
		In particular, the Hilbert series of $\B(M^{x,a})$ is	$(2)_{t}^2(3)_{t}^2$.
		\begin{center}
		\includegraphics[scale=.6]{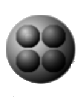}
		\end{center}
	  \item There are $6$ vertices with two orthogonally intersecting hyperplanes (the 
		subdivision centers of the cuboctahedron squares) as in the examples $a'$ the bounding 
		hyperplanes $x=H_{\alpha_{12}}\cap H_{\alpha_{23}}$. It is 
		associated to the following restricted Nichols algebra $\B(M^{x,a})$ in the category 
		$$\CC_x=\YD{\B(M_{\alpha_{12}}\oplus M_{\alpha_{23}})}, \mbox{ where }$$
		$$\B(M_{\alpha_{12}}\oplus M_{\alpha_{23}})\cong \C[t_{\alpha_{12}}]/(t_{
		\alpha_{12}}^3) \otimes \C[t_{\alpha_{23}}]/(t_{\alpha_{23}}^3)$$ 
		is a Nichols algebra of type 
		$A_1\times A_1$. The generating 
		Yetter-Drinfel'd modules $M^{x,a}$ is simple, the Nichols algebra of non-reduced type $
		A_1$: 
		$$M^{x,a}=M_{-\alpha_{2}}\oplus M_{\alpha_{12}}
		\oplus M_{\alpha_{23}}\oplus M_{\alpha_{123}}\qquad
		[M^{x,a},M^{x,a}]=M_{\alpha_{13}}$$
		In particular, the Hilbert series of $\B(M^{x,a})$ is  
		$(2)_{t}^4(3)_{t^2}$.
		\begin{center}
		\includegraphics[scale=.6]{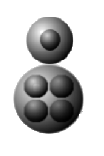}
		\end{center}
	\end{itemize}
	\item Again the $(-1)$-cell $x=0$ in the center is the trivial Nichols algebra $1\C$ 
	over any Weyl representative of the full Nichols algebra $\CC_x=\YD{\B(M^a)}$.\\
\end{itemize}

\def\cprime{$'$}
\providecommand{\bysame}{\leavevmode\hbox to3em{\hrulefill}\thinspace}
\providecommand{\MR}{\relax\ifhmode\unskip\space\fi MR }
\providecommand{\MRhref}[2]{%
  \href{http://www.ams.org/mathscinet-getitem?mr=#1}{#2}
}
\providecommand{\href}[2]{#2}

\end{document}